\documentclass[a4paper,12pt,intlimits,oneside,reqno]{amsart}
\usepackage{enumerate}
\usepackage{amsfonts,amsmath}

\usepackage{latexsym,amssymb}
\newtheorem{theorem}{Theorem}[section]

\newtheorem{lemma}[theorem]{Lemma}

\theoremstyle{corollary}

\theoremstyle{proposition}
\newtheorem{proposition}[theorem]{Proposition}

\theoremstyle{definition}
\newtheorem{definition}[theorem]{Definition}

\theoremstyle{remark}

\numberwithin{equation}{section}

\newcommand{\comment}[1]{}

\begin{document}

\title [Rough Hausdorff operators on the Heisenberg group]{Weighted norm inequalities for rough Hausdorff operator and its commutators
\\
on the Heisenberg group}

\thanks{This paper is funded by Vietnam National Foundation for Science and Technology Development (NAFOSTED)}

\author{Nguyen Minh Chuong}
\address{Institute of mathematics, Vietnamese Academy of Science and Technology, Hanoi, Vietnam}
\email{nmchuong@math.ac.vn}

\author{Dao Van Duong}
\address{School of Mathematics, Mientrung University of Civil Engineering, Phu Yen, Vietnam}
\email{daovanduong@muce.edu.vn}

\author{Kieu Huu Dung}
\address{School of Mathematics, University of Transport and Communications, Ha Noi, Vietnam}
\email{khdung@utc2.edu.vn}
\begin{abstract}
The aim of this paper is to study the sharp bounds of rough Hausdorff operators on the product of Herz, central Morrey and Morrey-Herz spaces with both power weights and  Muckenhoupt weights on the Heisenberg group. Especially, by applying the block decomposition of the Herz space, we obtain the boundedness of rough Hausdorff  operator in the case $0<p<1$. In addition,  the boundedness for the commutators of rough Hausdorff operators on such spaces with symbols in weighted central BMO space is also established.
\end{abstract}

\maketitle

\section{Introduction}
Let $\Phi$ be an integrable function on the positive half-line. 
The one dimensional Hausdorff operator associated to the function $\Phi$ is then defined by
$$ H_{\Phi}(f)(x)=\int\limits_{0}^\infty{\frac{\Phi(y)}{y}f\left(\frac{x}{y}\right)dy}.$$ 
 The Hausdorff operator may be originated by Hurwitz and Silverman \cite{Hurwitz} in order to study summability of number series (see also \cite{Hausdorff}). It is well known that the Hausdorff operator is one of important operators in harmonic analysis, and it is used to solve certain classical problems in analysis (see, for instance, \cite{Andersen2}, \cite{BM}, \cite{CFL2012}, \cite{Liflyand1}, \cite{Liflyand2} and the references therein). 
\vskip 5pt
In 2002, Brown and M\'{o}ricz \cite{BM} extended the study of  Hausdorff operator to the high dimensional space.
Given $\Phi$ be a locally integrable function on $\mathbb R^n$, the $n$-dimensional Hausdorff operator $H_{\Phi,A}$ associated to the kernel function $\Phi$ is then defined in terms of the integral form as follows
\begin{equation}\label{Hausdorff1}
H_{\Phi, A}(f)(x)=\int\limits_{\mathbb R^n}{\frac{\Phi(y)}{|y|^n}f(A(y) x)dy},\,x\in\mathbb R^n,
\end{equation}
where $A(y)$ is an $n\times n$ invertible matrix for almost everywhere $y$ in the support of $\Phi$. 
It is worth pointing out that if the kernel function $\Phi$ is taken appropriately, then  the Hausdorff operator reduces to many classcial operators in analysis such as the Hardy operator, the Ces\`{a}ro operator, the Riemann-Liouville fractional integral operator and the Hardy-Littlewood average operator (see, e.g., \cite{Andersen1}, \cite{Chuong2016},  \cite{CDH2016}, \cite{Christ}, \cite{Dzherbashyan} \cite{FGLY2015}, \cite{Miyachi}, \cite{Moricz2005}, \cite{Xiao} and references therein).
\vskip 5pt
In  2012,  Chen,  Fan and  Li \cite{CFL2012} introduced the rough Hausdorff operator on $\mathbb R^n$ which is defined as follows
\begin{equation}\label{RoughHausdorf}
{\widetilde {H}}_{\Phi,\Omega}(f)(x)=\int_{\mathbb R^n}\dfrac{\Phi(x|y|^{-1})}{|y|^n}\Omega(y')f(y)dy,\;\;x\in\mathbb R^n,
\end{equation}
where $y'=\frac{y}{|y|}$.  It is useful to remark that if $\Phi$ is a radial function, then for $\Phi(t)=t^{-n}\chi_{(1,\infty)}(t)$ and $\Omega \equiv 1$, the rough Hausdorff operator $\widetilde {H}_{\Phi,\Omega}$ reduces to the famous Hardy operator
 \begin{equation}\label{hardyoperator}
{{\mathcal H}}(f)(x)=\frac{1}{|x|^n}\int_{|y|\leq |x|}f(y)dy.
\end{equation}
Also, if  $\Omega \equiv 1$ and $\Phi(t)=\chi_{(0, 1)}(t)$, then $\widetilde {H}_{\Phi,\Omega}$ reduces to the adjoint Hardy operator 
 \begin{equation}\label{hardyoperator}
{{\mathcal H^\star}}(f)(x)=\int_{|y|> |x|} \frac{f(y)}{|y|^n}dy.
\end{equation}
\vskip 5pt
It is important to see that Hausdorff operator is focused on the $n$-dimensional Euclidean spaces $\mathbb R^n$ (see \cite{Andersen2}, \cite{CFL2012}, \cite{CDH2016}, \cite{Christ}, \cite{FGLY2015}, \cite{Fu}, \cite{Miyachi}, \cite{Moricz2005}, \cite{Tang}).  It is of interest to know whether the techniques for the investigation of Hausdorff operator in $\mathbb R^n$ can be used in different underlying spaces.  As it is known, there are many works studying on $p$-adic Hardy operator, Hardy-Littlewood average operator, $p$-adic Hardy-Ces\`{a}ro operator, $p$-adic Hausdorff operator as well as some their applications (see \cite{Chuong2016}, \cite{Chuongduong}, \cite{FWL2013}, \cite{Hung2014},  \cite{Rim},  \cite{Volosivets3} and references therein).
\vskip 5pt
It is also interesting to see that the theory of the Heisenberg group plays an important role in several branches of mathematics such as harmonic analysis, representation theory, several complex analysis, partial differential equation and quantum mechanics. For further readings on the theory of the Heisenberg group as well as its deep applications, one may find in the famous books \cite{Stein1993} and \cite{Thangavelu1998} for more details.  Let us first recall some basic about the Heisenberg group. The Heisenberg group $\mathbb H^n$ is a non-commutative nilpotent Lie group associated with underlying manifold
$\mathbb R^{2n}\times \mathbb R$ with group law as follows: 
for $x = (x_1, . . ., x_{2n}, x_{2n+1})$, $y = (y_1,..., y_{2n}, y_{2n+1})$ $\in \mathbb R^{2n}\times\mathbb R$,
$$
x\cdot y = \Big(x_1 + y_1, x_2 + y_2, . . ., x_{2n} + y_{2n}, x_{2n+1} + y_{2n+1} + 2 \sum\limits_{j=1}^n(y_jx_{n+j}- x_jy_{n+j})\Big).
$$
It is remarkable that this multiplication is non-commutative. It is obvious that by the definition, the identity element on the Heisenberg group is $0\in \mathbb R^{2n+1}$, and the reverse element of $x$ is $-x$. The vector fields
$$ X_j= \frac{\partial}{\partial {x_j}}+2x_{n+j}\frac{\partial}{\partial {x_{2n+j}}},\;\;\;j=1,...,n,$$
$$ X_{n+j}= \frac{\partial}{\partial {x_{n+j}}}-2x_{j}\frac{\partial}{\partial {x_{2n+1}}},\;\;\;j=1,...,n,$$
$$ X_{2n+1}= \frac{\partial}{\partial {x_{2n+1}}},$$
form a natural basis for the Lie algebra of left invariant vector fields. The only nontrivial commutator relations are
$$ [X_j, X_{n+j}]=-4X_{2n+1},\;\;j=1,2,...,n. $$
The Heisenberg group $\mathbb H^n$ is a space of homogeneous type in the sense of Coifman and Weiss (see \cite{Coifman}) with dilations
$$
\delta_rx = (rx_1, rx_2, . . ., rx_{2n}, r^2x_{2n+1}), \,\text{ for all }r > 0.
$$
For any measurable set $E \subseteq\mathbb H^n$, let us denote by $|E|$ the measure of $E$.  Then, one has
$$
|\delta_r(E)| = r^Q|E|,\,\,\, d(\delta_rx) = r^Qdx,
$$
where $Q = 2n + 2$ is the so-called homogeneous dimension. And, we have the norm
$$
|x|_h =\left(\Big(\sum\limits_{i=1}^{2n}x_i^2\Big)^2+ x^2_{2n+1}\right)^{\frac{1}{4}}.
$$
The distance on the Heisenberg group $\mathbb H^n$ derived from the norm above
is defined by
$$
d(x,y) = |y^{-1}\cdot x|_h,
$$
where $y^{-1}$ is the inverse of $y$. For $x\in\mathbb H^n, r > 0$, the ball with center $x$ and radius $r$ on $\mathbb H^n$ is given by
$$
B(x, r) =\{y\in\mathbb H^n : d(x,y) < r\},
$$
$$ B'(x, r) =\{y\in\mathbb H^n : d(x,y) \leq r\},$$
and its sphere is defined by
$$
S(x, r) = \{y\in\mathbb H^n : d(x,y) = r\}.
$$
It is evident that
$$
|B(x,r)| = |B(0,r)| = \nu_Qr^Q,
$$
where $\nu_Q$ is the volume of the unit ball $B(0,1)$ on $\mathbb H^n$ with
$$
\nu_Q =\dfrac{2\pi^{n+ 1/2}\Gamma(\frac{n}{2})}{\Gamma(n + 1)\Gamma(\frac{n+1}{2})}.
$$
The unit sphere $S(0,1)$ is often simply denoted by $S_{Q-1}$, and the area of $S_{Q-1}$ is $\omega_Q = Q\nu_Q.$
The reader is referred to \cite{Thangavelu1998} for more details.
\vskip 5pt
In recent years, many mathematicians have been interested in studying the Hardy operators, weighted Hardy operators and Hausdorff operators in the setting of the Heisenberg group (see \cite{Chu}, \cite{Guo}, \cite{RFW2017}, \cite{RFW2017-1}, \cite{Wu2016}, \cite{WF2017}). More details,  Ruan, Fan and Wu \cite{RFW2017} introduced and studied Hausdorff type operator on the Heisenberg group defined as follows
\begin{align}
H_{\Phi, A}(f)(x)=\int\limits_{\mathbb H^n}{\frac{\Phi(y)}{|y|_h^Q}f(A(y) x)dy},\,x\in\mathbb H^n.
\end{align} 
In addition, the fractional Hausdorff operator on $\mathbb H^n$ \cite{WF2017} is also defined by
\begin{align}
T_{\Phi,\beta}(f)(x)=\int\limits_{\mathbb H^n}{\frac{\Phi(\delta_{|y|_h^{-1}}x)}{|y|_h^{Q-\beta}}f(y)dt},\,0\leq \beta\leq Q,
\end{align} 
an  the boundedness of $T_{\Phi,\beta}$ on Hardy spaces is obtained.
\vskip 5pt
Now, it is actually natural to try to extend and study the rough type Hausdorff operator in the setting of the Heisenberg group. Namely, we have the following definition for rough Hausdorff operator on $\mathbb H^n$.
\begin{definition}
Let $\Phi:\mathbb H^n\longrightarrow [0, \infty) $ be a radial measurable function, and let $\Omega:S_{Q-1}\longrightarrow\mathbb C $ be a measurable function such that $\Omega(y)\neq 0$ for almost everywhere $y$ in $S_{Q-1}$.  Let $f$ be measurable complex-valued functions on $\mathbb H^n$. The rough Hausdorff operator  is then defined by
\begin{equation}\label{RH}
{\mathcal H}_{\Phi,\Omega}(f)(x)=\int_{\mathbb H^n} \dfrac{\Phi(\delta_{|y|_h^{-1}}x)}{|y|_h^{Q}}\Omega(\delta_{|y|_h^{-1}}y)f(y)dy,\;\;x\in\mathbb H^n.
\end{equation}
Remark that using polar coordinates in the Heisenberg group, we have 
\begin{equation}\label{RH1}
{\mathcal H}_{\Phi,\Omega}(f)(x)=\int_{0}^{\infty}\int_{S_{Q-1}} \dfrac{\Phi(t)}{t}\Omega(y')f(\delta_{t^{-1}|x|_h}y')dy'dt,\;\;x\in\mathbb H^n.
\end{equation}
\end{definition} 
It turns out  that for $\Phi(t)=t^{-Q}\chi_{(1,\infty)}(t)$ and $\Omega \equiv 1$, the rough Hausdorff operator ${\mathcal H}_{\Phi,\Omega}$ reduces to the Hardy operator on the Heisenberg group. Also, if $\Phi(t)=\chi_{(0, 1)}(t)$ and $\Omega \equiv 1$, then ${\mathcal H}_{\Phi,\Omega}$ reduces to the adjoint Hardy operator. The reader may refer to \cite{Wu2016} for the definition of the Hardy operator and adjoint Hardy operator on the Heisenberg group.
\vskip 5pt
Let $b$ be a locally integrable function. We denote by $\mathcal{M}_b$ the multiplication operator defined  by $\mathcal{M}_bf (x)=b(x) f (x)$ for any measurable function $f$. If $\mathcal{H}$ is a linear operator on some measurable function space, the commutator of Coifman-Rochberg-Weiss type formed by $\mathcal{M}_b$  and $\mathcal{H}$ is defined by $[\mathcal{M}_b, \mathcal{H}]f (x)=(\mathcal{M}_b\mathcal{H}-\mathcal{H}\mathcal{M}_b) f (x)$. Next, let us give the definition for the commutators of Coifman-Rochberg-Weiss type of rough Hausdorff operator on the Heisenberg group.
\begin{definition}
Let $\Phi, \Omega$ be as above. The Coifman-Rochberg-Weiss type commutator of rough Hausdorff operator is defined by
\begin{equation}\label{commuatatorRH}
{\mathcal H}_{\Phi,\Omega}^b(f)(x)=\int_{0}^{\infty}\int_{S_{Q-1}} \dfrac{\Phi(t)}{t}\Omega(y')(b(x)-b(\delta_{t^{-1}|x|_h}y'))f(\delta_{t^{-1}|x|_h}y')dy'dt, 
\end{equation}
where $b$ is locally integrable functions on $\mathbb H^n$  and $x\in\mathbb H^n$.
\end{definition}
\vskip 5pt
Inspired by above mentioned results, the main purpose of this paper is to establish the necessary and sufficient conditions for  the boundedness of ${\mathcal {H}}_{\Phi,\Omega}$ on the product of Herz, central Morrey and Morrey-Herz with both power weights and Muckenhoupt weights in the Heisenberg group. In each case, we estimate the corresponding operator norms. Moreover, we also obtain the boundedness of ${\mathcal H}_{\Phi,\Omega}^b$ on such spaces with symbols in weighted central BMO space.
\vskip 5pt
 Our paper is organized as follows. In Section \ref{2}, we present some notations and definitions of the Herz, central Morrey, Morrey-Herz and central BMO spaces associated with the weights on the Heisenberg group. Our main theorems are given and proved in Section \ref{3} and Section \ref{4}.
\section{Some notations and definitions}\label{2}
Throught the whole paper, we denote by $C$ a positive geometric constant that is independent of the main parameters, but can change from line to line. We also write $a \lesssim b$ to mean that there is a positive constant $C$, independent of the main parameters, such that $a\leq Cb$. The symbol $f\simeq g$ means that $f$ is equivalent to $g$ (i.e.~$C^{-1}f\leq g\leq Cf$). By $\|T\|_{X\to Y}$, we denote the norm of $T$ between two normed vector spaces $X$ and $Y$. Let $L^q_{\omega}(\mathbb{H}^n)$ $(0<q<\infty)$ be the space of all measurable functions $f$ on $\mathbb{H}^n$ such that
\begin{align*}
\|f\|_{L^q_\omega(\mathbb H^n)}=\Big(\,\int_{\mathbb{H}^n}|f(x)|^q\omega(x)dx \Big)^{\frac{1}{q}}<\infty.
\end{align*} 
The space $L^q_\text {loc}(\omega, \mathbb H^n)$ is defined as the set of all measurable functions $f$ on $\mathbb H^n$ satisfying $\int_{K}|f(x)|^q\omega(x)dx<\infty$ for any compact subset $K$ of $\mathbb H^n$. The space $L^q_\text {loc}(\omega, \mathbb H^n\setminus\{0\})$ is also defined in a similar way to the space  $L^q_\text {loc}(\omega, \mathbb H^n)$.
\vskip 5pt
In what follows, we denote $\chi_k=\chi_{C_k}$, $C_k=U_k\setminus U_{k-1}$, where $U_k = \big\{x\in \mathbb H^n: |x|_h \leq 2^k\big\}$ for all $k\in\mathbb Z$. We also denote $B_R=\{x\in\mathbb H^n: |x|_h\leq R\}$
for all $R > 0$.  As usual, the weighted function $\omega$ is a non-negative measurable function on $\mathbb H^n$. Denote by $\omega(K)$ the integral $\int_{K}\omega(x)dx$ for all subsets $K$ of $\mathbb H^n$.
\vskip 5pt
Now, we are in a position to give some notations and definitions of  Morrey, Herz and Morrey-Herz spaces with weights on the Heisenberg group.
\begin{definition} Let $1\leq q<\infty,\frac{-1}{q}<\lambda<0$.
The weighted $\lambda$-central Morrey spaces ${\mathop{B}\limits^.}^{q,\lambda}_{\omega}(\mathbb H^n)$ consists of all measurable functions $f\in L^q_{\omega,\rm loc}(\mathbb H^n\setminus \{0\})$ satisfying
\begin{equation}
\|f\|_{{\mathop{B}\limits^.}^{q,\lambda}_{\omega}(\mathbb H^n)} =\mathop{\rm sup}\limits_{R\in \mathbb R^+}\Big(\dfrac{1}{\omega(B_R)^{1+\lambda q}}\int_{B_R}|f(x)|^q\omega(x)dx\Big)^{1/q}<\infty.
\end{equation}
\end{definition}
\begin{definition} Let $\alpha\in\mathbb R, 0<q<\infty$ and $0<p <\infty$. The weighted Herz space  $K_{q,\omega}^{\alpha,p}(\mathbb H^n)$ is defined as the set of all  measurable functions $f\in L^q_{\omega,\rm loc}(\mathbb H^n\setminus \{0\})$  such that
\begin{equation}
\|f\|_{K_{q,\omega}^{\alpha,p}(\mathbb H^n)} =\Big( \sum\limits_{k=-\infty}^{\infty} 2^{k\alpha p}\|f\chi_k\|_{L^q_\omega(\mathbb H^n)}^{p}\Big)^{1/p}<\infty.
\end{equation}
\end{definition}
We also consider two Herz-type spaces.
\begin{definition} Let $\alpha\in\mathbb R, 0<q<\infty$ and $0<p <\infty$. The weighted Herz space  ${\mathop{K}\limits^{.}}_{q,\omega}^{\alpha,p}(\mathbb H^n)$ is defined as the set of all measurable functions $f\in L^q_{\omega,\rm loc}(\mathbb H^n\setminus \{0\})$  such that $\|f\|_{{\mathop{K}\limits^{.}}_{q,\omega}^{\alpha,p}(\mathbb H^n)}<\infty$,
where
\begin{equation}
\|f\|_{{\mathop{K}\limits^{.}}_{q,\omega}^{\alpha,p}(\mathbb H^n)}=\Big( \sum\limits_{k=-\infty}^{\infty} \omega(U_k)^{\alpha p/Q}\|f\chi_k\|_{L^q_\omega(\mathbb H^n)}^{p}\Big)^{1/p}.
\end{equation}
\end{definition}
\begin{definition} Let $\alpha\in\mathbb R, 0<q<\infty$, $0<p <\infty$ and $\lambda$ be a non-negative real number. The weighted Morrey-Herz space is defined by
\[
MK_{p, q,\omega}^{\alpha,\lambda}(\mathbb H^n)=\left\{f\in L^q_{\omega,\rm loc}(\mathbb H^n\setminus \{0\}):\|f\|_{MK_{p, q,\omega}^{\alpha,\lambda}(\mathbb H^n)}<\infty\right\},
\] 
where
\begin{equation}
\|f\|_{MK_{p, q,\omega}^{\alpha,\lambda}(\mathbb H^n)}=\mathop{\rm sup}\limits_{k_0\in\mathbb Z}2^{-k_0\lambda}\Big( \sum\limits_{k=-\infty}^{k_0} 2^{k\alpha p}\|f\chi_k\|_{L^q_\omega(\mathbb H^n)}^{p}\Big)^{1/p}.
\end{equation}
\end{definition}
Let us recall to define the weighted central BMO space.
\begin{definition}
Let $1\leq q<\infty$ and $\omega$ be a weighted function. The weighted central bounded mean oscillation space ${{CMO}}^q_\omega(\mathbb H^n)$ is defined as the set of all functions $f\in L^q_{\omega,\rm loc}(\mathbb H^n)$ such that
\begin{equation}
\big\|f\big\|_{{{CMO}}^q_\omega(\mathbb H^n)}=\mathop {\rm sup}\limits_{R>0}\Big( \frac{1}{\omega(B_R)}\int\limits_{B_R}{|f(x)-f_{B_R}|^q\omega(x)dx}\Big)^{\frac{1}{q}},
\end{equation}
where 
\[
f_{B_R}=\frac{1}{|B_R|}\int\limits_{B_R}{f(x)dx}.
\]
\end{definition} 
\vskip 5pt
It is well known that the theory of $A_p$ weight was first introduced by Muckenhoupt \cite{Muckenhoupt1972} in the Euclidean spaces in order to characterise  the weighted $L^p$ boundedness of Hardy-Littlewood maximal functions. For $A_p$ weights  on the homogeneous type spaces,  one can refer to the work \cite{HCE2012}. Remark that the Heisenberg group is also a space of homogeneous type space. Next, let us recall some basic properties about $A_p$ weights on the Heisenberg group which are used in the sequel.
\begin{definition} Let $1 < p < \infty$. It is said that a weight $\omega \in A_p(\mathbb H^n)$ if there exists a constant $C$ such that for all balls $B$,
$$\Big(\dfrac{1}{|B|}\int_{B}\omega(x)dx\Big)\Big(\dfrac{1}{|B|}\int_{B}\omega(x)^{-1/(p-1)}dx\Big)^{p-1}\leq C.$$
It is said that a weight $\omega\in A_1(\mathbb H^n)$ if there is a constant $C$ such that for all balls $B$,
$$\dfrac{1}{|B|}\int_{B}\omega(x)dx\leq C\mathop{\rm essinf}\limits_{x\in B}\omega(x).$$
We denote $A_{\infty}(\mathbb H^n) = \bigcup\limits_{1\leq p<\infty}A_p(\mathbb H^n)$.
\end{definition}
\begin{proposition} The following statements are true:
\begin{enumerate}
\item[(i)] $A_p(\mathbb H^n)\subsetneq A_q(\mathbb H^n)$, for $1\leq  p < q < \infty$.
\item[(ii)] If $\omega\in A_p(\mathbb H^n)$, $1 < p < \infty$, then there is an $\varepsilon > 0$ such that $p-\varepsilon > 1$ and $\omega\in A_{p-\varepsilon}(\mathbb H^n)$.
\end{enumerate}
\end{proposition}
\vskip 5pt
A close relation to $A_{\infty}(\mathbb H^n)$ is the reverse H\"{o}lder condition. If there exist $r > 1$ and a fixed constant $C$ such that
$$\Big(\dfrac{1}{|B|}\int_{B}\omega(x)^rdx\Big)^{1/r}\leq \dfrac{C}{|B|}\int_{B}\omega(x)dx,$$
for all balls $B \subset\mathbb H^n$, we then say that $\omega$ satisfies the reverse H\"{o}lder condition of order $r$ and write $\omega\in RH_r(\mathbb H^n)$. According to Theorem 19 and Corollary 21 in \cite{IMS2015}, $\omega\in A_{\infty}(\mathbb H^n)$ if and
only if there exists some $r > 1$ such that $\omega\in RH_r(\mathbb H^n)$. Moreover, if $\omega\in RH_r(\mathbb H^n)$, $r > 1$, then $\omega\in RH_{r+\varepsilon}(\mathbb H^n)$ for some $\varepsilon > 0$. We thus write $r_\omega \equiv {\rm sup}\{r > 1: \omega\in RH_r(\mathbb H^n)\}$ to denote the critical index of $\omega$ for the reverse H\"{o}lder condition. For further properties of $A_p$ weights, one may find in the books  \cite{LDY2007} and \cite{Stein1993}.
\vskip 5pt
It is useful to note that an important example of $A_p(\mathbb H^n)$ weights is the power function $|x|_h^{\alpha}$. By the similar proofs to Propositions 1.4.3 and 1.4.4 in \cite{LDY2007} , we also obtain the following properties of power weighted functions.
\begin{proposition} The following statements are true:
\begin{itemize}
\item[(i)] $|x|^{\alpha}_h\in A_1(\mathbb H^n)$ if and only if $-Q< \alpha\leq 0$;
\item[(ii)] $|x|^{\alpha}_h\in A_p(\mathbb H^n)$, $1 < p < \infty$, if and only if $-Q < \alpha < Q(p-1)$.
\end{itemize}
\end{proposition}
Let us have the following standard characterization of $A_p$ weights which is the same as the real setting  (see \cite{RFW2017} for more details).
\begin{proposition}\label{pro2.3DFan}
Let $\omega\in A_p(\mathbb H^n) \cap RH_r(\mathbb H^n)$, $p\geq 1$ and $r > 1$. Then there exist constants $C_1, C_2 > 0$ such that
$$
C_1\left(\dfrac{|E|}{|B|}\right)^p\leq \dfrac{\omega(E)}{\omega(B)}\leq C_2\left(\dfrac{|E|}{|B|}\right)^{(r-1)/r}
$$
for any measurable subset $E$  of a ball $B$.
\end{proposition}
\begin{proposition}\label{pro2.4DFan}
If $\omega\in A_p(\mathbb H^n)$, $1 \leq p < \infty$, then for any $f\in L^1_{\rm loc}(\mathbb H^n)$ and any ball $B \subset \mathbb H^n$,
$$
\dfrac{1}{|B|}\int_{B}|f(x)|dx\leq C\left(\dfrac{1}{\omega(B)}\int_{B}|f(x)|^p\omega(x)dx\right)^{1/p}.
$$
\end{proposition}
\begin{definition}
Let $0<\alpha< \infty$ and $1\leq q <\infty$. A function $b$ on $\mathbb H^n$ is said
to be a $\textit{central}\,(\alpha, q, \omega)-\textit{block}$ if it satisfies
\[
\textit{\rm supp}(b)\subset B(0,r)\,\,\,\textit{\rm and}\,\,\,\|b\|_{L^q_{\omega}(\mathbb H^n)}\leq \omega(B(0,r))^{-\alpha/Q}.
\]
\end{definition}
 The following useful decomposition theorem follows us to show that the central blocks are  the  building blocks of Herz spaces. The proof  is the same as that of Theorem 1.1 in \cite{LY1995}. 
\begin{proposition}\label{blockHerz}
Let $0 <\alpha < \infty, 0 < p <\infty, 1\leq q <\infty$, and let $\omega\in A_1(\mathbb H^n)$.
We then have  $f\in {\mathop{K}\limits^{.}}^{\alpha,p}_{q,\omega}(\mathbb H^n)$  if and only if
\[
f=\sum\limits_{k=-\infty}^{\infty}\lambda_k b_k,
\]
where $\sum\limits_{k=-\infty}^{\infty}{|\lambda_k|^p}<\infty$, and each $b_k$ is a central $(\alpha, q, \omega)$-block with the support
in $U_k$. Moreover,
\[
\|f\|_{{\mathop{K}\limits^{.}}^{\alpha,p}_{q,\omega}(\mathbb H^n)}\simeq \textit{\rm inf}\Big\{\Big(\sum\limits_{k=-\infty}^{\infty}|\lambda_k|^p\Big)\Big\}^{1/p},
\]
where the infimum is taken over all decompositions of f as above.
\end{proposition}
\section{The main results about the boundness of ${\mathcal{H}}_{\Phi,\Omega}$}\label{3}
Now let us give our first main results concerning the boundedness of rough Hausdorff operators on the weighted Morrey spaces on the Heisenberg group.
\begin{theorem}\label{TheoremMorrey}
Let $1< q<\infty$, $\lambda \in \big(\frac{-1}{q},0\big)$, $\gamma\in (-Q,\infty)$, $\omega(x)= |x|^\gamma_h$ and $\Omega\in L^{q'}(S_{Q-1})$. Then,  ${\mathcal H}_{\Phi,\Omega}$ is bounded from ${\mathop B\limits^.}^{q,\lambda}_{\omega}(\mathbb H^n)$ to ${\mathop B\limits^.}^{q,\lambda}_{\omega}(\mathbb H^n)$ if and only if 
$$
\mathcal C_{1}= \int_{0}^{\infty} {\dfrac{\Phi(t)}{t^{1+(Q+\gamma)\lambda}}}dt <  + \infty.
$$
Moreover, ${\big\|{\mathcal H}_{\Phi,\Omega}\big\|_{{\mathop B\limits^.}^{{q},{\lambda}}_{{\omega}}(\mathbb H^n)\to {\mathop B\limits^.}^{q,\lambda }_\omega(\mathbb H^n)}}\simeq \mathcal C_{1}.\|\Omega\|_{L^{q'}(S_{Q-1})}.$
\end{theorem}
\begin{proof}
The first step is to prove the sufficient condition of this theorem.
For $R\in\mathbb R^+$, by the Minkowski inequality and the H\"{o}lder inequality we get
\begin{align}\label{HfMorrey0}
\|{\mathcal H}_{\Phi,\Omega}(f)\|_{L^q_\omega(B_R)}&\leq \int_{0}^{\infty}\dfrac{\Phi(t)}{t}\Big(\int_{B_R}\Big|\int_{S_{Q-1}}\Omega(y') f(\delta_{t^{-1}|x|_h}y')\omega^{1/q}(x)dy'\Big|^qdx\Big)^{1/q}dt\nonumber
\\
&\leq \int_{0}^{\infty}\dfrac{\Phi(t)}{t}\Big(\int_{S_{Q-1}}|\Omega(y)|\Big(\int_{B_R}|f(\delta_{t^{-1}|x|_h}y')|^q\omega(x)dx\Big)^{1/q}dy'\Big)dt\nonumber.
\\
&\leq \|\Omega\|_{L^{q'}(S_{Q-1})}\int_{0}^{\infty}\dfrac{\Phi(t)}{t}\Big(\int_{S_{Q-1}}\int_{B_R}|f(\delta_{t^{-1}|x|_h}y')|^q\omega(x)dxdy'\Big)^{1/q}dt\nonumber.
\\
\end{align}
On the other hand, by applying polar coordinates, we have
\begin{align}\label{S0BR}
&\int_{S_{Q-1}}\int_{B_R}|f(\delta_{t^{-1}|x|_h}y')|^q\omega(x)dxdy'=\int_{S_{Q-1}}\int_{0}^{R}\int_{S_{Q-1}}|f(\delta_{t^{-1}|\delta_{r}u'|_h}y')|^q|\delta_r u'|_h^{\gamma}r^{Q-1}du'drdy'\nonumber
\\
&=\int_{S_{Q-1}}\int_{0}^{R}\int_{S_{Q-1}}|f(\delta_{t^{-1}r}y')|^q r^{Q-1+\gamma}du'drdy'\simeq \int_{S_{Q-1}}\int_{0}^{R}|f(\delta_{t^{-1}r}y')|^q r^{Q-1+\gamma}drdy'\nonumber
\\
&=t^{Q+\gamma}\int_{0}^{t^{-1}R}\int_{S_{Q-1}}|f(\delta_{z}y')|^q z^{Q-1+\gamma}dzdy'=t^{Q+\gamma}\int_{B_{t^{-1}R}}|f(x)|^q \omega(x)dx.
\end{align}
This implies that
\begin{align}\label{HfLebMorrey}
\|{\mathcal H}_{\Phi,\Omega}(f)\|_{L^q_\omega(B_R)}\lesssim \|\Omega\|_{L^{q'}(S_{Q-1})}\int_{0}^{\infty}\dfrac{\Phi(t)}{t^{1-\frac{Q+\gamma}{q}}}\|f\|_{L^q_\omega(B_{t^{-1}R})}dt.
\end{align}
Next, we need to show that
\begin{align}\label{omegaBt-1R}
\dfrac{1}{\omega(B_R)^{1/q+\lambda}}\simeq  \dfrac{t^{-(Q+\gamma)(1/q+\lambda)}}{\omega (B_{t^{-1}R})^{1/{q}+\lambda}}.
\end{align}
Indeed, we get
\begin{align}\label{omegaBr}
&\omega(B_R)=\int_{B_R}|x|_h^{\gamma}dx=\int_{0}^{R}\int_{S_{Q-1}}|\delta_{ry'}|_h^{\gamma}r^{Q-1}dy'dr=\int_{0}^{R}\int_{S_{Q-1}}r^{Q-1+\gamma}dy'dr\simeq R^{Q+\gamma},
\nonumber
\\
&\textit{\rm and}\,\, \omega(B_{t^{-1}R})\simeq (t^{-1}R)^{Q+\gamma}.
\end{align}
Hence, the proof of estimate (\ref{omegaBt-1R}) is finished.
From (\ref{HfLebMorrey}) and (\ref{omegaBr}), for $R\in\mathbb R$, it is easy to see that
$$
\dfrac{1}{\omega(B_R)^{1/q+\lambda}}\|{\mathcal H}_{\Phi,\Omega}(f)\|_{L^q_\omega(B_R)}\lesssim \|\Omega\|_{L^{q'}(S_{Q-1})}\Big(\int_{0}^{\infty}\dfrac{\Phi(t)}{t^{1+(Q+\gamma)\lambda}}dt\Big)\|f\|_{{\mathop B\limits^.}^{q,\lambda}_{\omega}(\mathbb H^n)}.
$$
Thus, by  the definition of the Morrey space, we obtain
$$\|{\mathcal H}_{\Phi,\Omega}(f)\|_{{\mathop B\limits^.}^{q,\lambda}_{\omega}(\mathbb H^n)}\lesssim \mathcal C_1.\|\Omega\|_{L^{q'}(S_{Q-1})}\|f\|_{{\mathop B\limits^.}^{q,\lambda}_{\omega}(\mathbb H^n)},
$$
which means that the operator ${\mathcal H}_{\Phi,\Omega}$ is bounded on ${\mathop B\limits^.}^{q,\lambda}_{\omega}(\mathbb H^n)$.
\vskip5pt
Now, the second step is to  assume that  ${\mathcal H}^p_{\Phi,\Omega}$ is bounded from ${\mathop B\limits^.}^{q,\lambda}_{\omega}(\mathbb H^n)$ to ${\mathop B\limits^.}^{q,\lambda}_{\omega}(\mathbb H^n)$. Let us choose an appropriately radial function as follows
$$
f(x)=|x|_h^{(Q+\gamma)\lambda}.|\Omega(\delta_{|x|_h^{-1}}x)|^{q'-2}.\overline{\Omega}(\delta_{|x|_h^{-1}}x),\,\textit{\rm for}\, x\in \mathbb H^n\setminus\{0\}.
$$
Then, we have
\begin{align}\label{normfi}
\|f\|^{q}_{L^{q}_{\omega}(B_R)}&=\int_{B_R}|x|_h^{(Q+\gamma)\lambda q+\gamma}.|\Omega(\delta_{|x|_h^{-1}}x)|^{q'}dx\nonumber
\\
&=\int_{0}^{R}\int_{S_{Q-1}}|\delta_ry'|_h^{(Q+\gamma)\lambda q+\gamma}.|\Omega(\delta_{|\delta_ry'|_h^{-1}}\delta_ry')|^{q'}r^{Q-1}dy'dr
\nonumber
\\
&=\|\Omega\|_{L^{q'}(S_{Q-1})}^{q'}\int_{0}^{R}r^{(Q+\gamma)\lambda q+\gamma+Q-1}dr\nonumber\\
&\simeq \|\Omega\|_{L^{q'}(S_{Q-1})}^{q'}R^{(Q+\gamma)(\lambda q+1)} .
\end{align}
Hence, by (\ref{omegaBr}) and (\ref{normfi}), one has
\begin{align}\label{fESMorrey}
0<\|f\|_{{\mathop B\limits^.}^{q,\lambda}_{\omega}(\mathbb H^n)}\simeq \|\Omega\|_{L^{q'}(S_{Q-1})}^{q'/q}<\infty.
\end{align}
 By taking $f$ as above, we have
\begin{align}
{\mathcal H}_{\Phi,\Omega}(f)(x)&=\int_{0}^{\infty}\int_{S_{Q-1}}\dfrac{\Phi(t)}{t}\Omega(y')|\delta_{t^{-1}|x|_h}y'|_h^{(Q+\gamma)\lambda}|\Omega(y')|^{q'-2}\overline{\Omega}(y')dy'dr\nonumber
\\
&=\|\Omega\|_{L^{q'}(S_{Q-1})}^{q'}\Big(\int_{0}^{\infty}\dfrac{\Phi(t)}{t^{1+(Q+\gamma)\lambda}}dt\Big)|x|_h^{(Q+\gamma)\lambda},\nonumber
\end{align}
which implies  that
$$
\|{\mathcal H}_{\Phi,\Omega}(f)\|_{{\mathop B\limits^.}^{q,\lambda}_{\omega}(\mathbb H^n)}= \mathcal C_1.\|\Omega\|_{L^{q'}(S_{Q-1})}^{q'}\||x|_h^{Q+\gamma}\|_{{\mathop B\limits^.}^{q,\lambda}_{\omega}(\mathbb H^n)}.
$$
By estimating as above, we also have $\||x|_h^{Q+\gamma}\|_{{\mathop B\limits^.}^{q,\lambda}_{\omega}(\mathbb H^n)}\simeq 1.$ Thus, by (\ref{fESMorrey}), it follows immediately that
$$
\|{\mathcal H}_{\Phi,\Omega}(f)\|_{{\mathop B\limits^.}^{q,\lambda}_{\omega}(\mathbb H^n)}\gtrsim \mathcal C_1.\|\Omega\|_{L^{q'}(S_{Q-1})}^{q'}\dfrac{\|f\|_{{\mathop B\limits^.}^{q,\lambda}_{\omega}(\mathbb H^n)}}{\|\Omega\|_{L^{q'}(S_{Q-1})}^{q'/q}}=\mathcal C_1.\|\Omega\|_{L^{q'}(S_{Q-1})}\|f\|_{{\mathop B\limits^.}^{q,\lambda}_{\omega}(\mathbb H^n)},
$$
As a consequence, we obtain that $\mathcal C_1 < \infty$, and the proof of the theorem is completed.
\end{proof}
\begin{theorem}\label{TheoremMorrey1}
Let $1\leq q^*,q, \zeta<\infty$, $\lambda\in (\frac{-1}{q^*},0)$, $\Omega\in L^{({\frac{q^*}{\zeta}})'}(S_{Q-1})$ and $\omega\in A_{\zeta}$ with the finite critical index $r_{\omega}$ for the reverse H\"{o}lder condition.
\\
If $q^*>q\zeta r_{\omega}/(r_{\omega}-1)$, $\delta\in (1,r_\omega)$ and 
 $$\mathcal C_2=\int_{1}^{\infty}\dfrac{\Phi(t)}{t^{1+Q\lambda}}dt+\int_{0}^{1}\dfrac{\Phi(t)}{t^{1+Q\lambda(\delta-1)/\delta}}dt<\infty,
 $$
then the operator ${\mathcal H}_{\Phi,\Omega}$ is bounded from ${\mathop B\limits^.}^{q^*,\lambda}_{\omega}(\mathbb H^n)$ to ${\mathop B\limits^.}^{q,\lambda}_{\omega}(\mathbb H^n)$.
\end{theorem}
\begin{proof}
It follows from the Minkowski inequality that
\begin{align}\label{HfMor1}
\|{\mathcal H}_{\Phi,\Omega}(f)\|_{L_\omega^{q}(B_R)} &\leq\int_{0}^{\infty}\dfrac{\Phi(t)}{t}\Big(\int_{S_{Q-1}}|\Omega(y')|\Big(\int_{B_R}|f(\delta_{t^{-1}|x|_h}y')|^{q}\omega(x)dx\Big)^{1/{q}}dy'\Big)dt.
\end{align}
By assuming that $q^*>q\zeta r_\omega/(r_\omega-1)$, we have $r\in(1,r_\omega)$ and $\frac{q^*}{\zeta}=q r'$. From this, by combining the H\"{o}lder inequality and the reverse H\"{o}lder condition, it is not hard to obtain that
\begin{align}\label{fq*}
&\Big(\int_{B_R}|f(\delta_{t^{-1}|x|_h}y')|^{q}\omega(x)dx\Big)^{1/{q}}\leq \Big(\int_{B_R}|f(\delta_{t^{-1}|x|_h}y')|^{\frac{q^*}{\zeta}}dx\Big)^{\frac{\zeta}{q^*}}\Big(\int_{B_R}\omega^r(x)dx\Big)^{\frac{1}{rq}}\nonumber
\\
&\,\,\,\lesssim \Big(\int_{B_R}|f(\delta_{t^{-1}|x|_h}y')|^{\frac{q^*}{\zeta}}dx\Big)^{\frac{\zeta}{q^*}}\omega(B_R)^{\frac{1}{q}}|B_R|^{\frac{-\zeta}{q^*}}.
\end{align}
Thus, by (\ref{HfMor1}) and (\ref{fq*}), we infer
\begin{align}\label{HfMor2}
\|{\mathcal H}_{\Phi,\Omega}(f)\|_{L_\omega^{q}(B_R)} &\lesssim \omega(B_R)^{\frac{1}{q}}|B_R|^{\frac{-\zeta}{q^*}}\times
\nonumber
\\
&\,\,\,\times\int_{0}^{\infty}\dfrac{\Phi(t)}{t}\Big(\int_{S_{Q-1}}|\Omega(y')|\Big(\int_{B_R}|f(\delta_{t^{-1}|x|_h}y')|^{\frac{q^*}{\zeta}}dx\Big)^{\frac{\zeta}{q^*}}dy'\Big)dt.
\end{align}
By making  the H\"{o}lder inequality and estimating as (\ref{S0BR}) above, we also get
\begin{align}
&\int_{S_{Q-1}}|\Omega(y')|\Big(\int_{B_R}|f(\delta_{t^{-1}|x|_h}y')|^{\frac{q^*}{\zeta}}dx\Big)^{\frac{\zeta}{q^*}}dy'\nonumber
\\
&\,\,\,\,\,\,\leq\|\Omega\|_{L^{(\frac{q^*}{\zeta})'}(S_{Q-1})}\Big(\int_{S_{Q-1}}\int_{B_R}|f(\delta_{t^{-1}|x|_h}y')|^{\frac{q^*}{\zeta}}dxdy'\Big)^{\frac{\zeta}{q^*}}\nonumber
\\
&\,\,\,\,\,\,\lesssim \|\Omega\|_{L^{(\frac{q^*}{\zeta})'}(S_{Q-1})}t^{\frac{\zeta Q}{q^*}}\|f\|_{L^{q^*/\zeta}(B_{t^{-1}R})}.\nonumber
\end{align}
From Proposition \ref{pro2.4DFan}, one has
$$
\|f\|_{L^{q^*/\zeta}(B_{t^{-1}R})}\lesssim |B_{t^{-1}R}|^{\frac{\zeta}{q^*}}\omega(B_{t^{-1}R})^{\frac{-1}{q^*}}\|f\|_{L_{\omega}^{q^*}(B_{t^{-1}R})}.
$$
A similar argument as (\ref{omegaBr}) above, we have $\dfrac{|B_{t^{-1}R}|}{|B_R)|}\simeq \dfrac{(t^{-1}R)^Q}{R^Q}=t^{-Q}.$
Therefore, by (\ref{HfMor2}), we obtain
\begin{align}\label{HfBRMorrey}
\|\mathcal H_{\Phi,\Omega}(f)\|_{L_{\omega}^{q}(B_R)}\lesssim \|\Omega\|_{L^{(\frac{q^*}{\zeta})'}(S_{Q-1})}\int_{0}^{\infty}\dfrac{\Phi(t)}{t}\dfrac{\omega(B_R)^{\frac{1}{q}}}{\omega(B_{t^{-1}R})^{\frac{1}{q^*}}}\|f\|_{L_{\omega}^{q^*}(B_{t^{-1}R})}dt.
\end{align}
Thus, by the definition of the Morrey space, we get
\begin{align}\label{HfMor3}
\|\mathcal H_{\Phi,\Omega}(f)\|_{\mathop{B}_{\omega}^{q,\lambda}(\mathbb H^n)}&\lesssim \|\Omega\|_{L^{(\frac{q^*}{\zeta})'}(S_{Q-1})}\|f\|_{\mathop{B}_{\omega}^{q^*,\lambda}(\mathbb H^n)}\int_{0}^{\infty}\dfrac{\Phi(t)}{t}\mathop{\rm sup}\limits_{R>0}\Big(\dfrac{\omega(B_{t^{-1}R})}{\omega(B_R)}\Big)^{\lambda}dt.
\end{align}
Next, since $\lambda<0$ and $\delta\in(1,r_\omega)$, by  Proposition \ref{pro2.3DFan}, we deduce
\begin{align}\label{t-1Rlambda}
\Big(\dfrac{\omega(B_{t^{-1}R})}{\omega(B_R)}\Big)^{\lambda}\lesssim \left\{ \begin{array}{l}
\Big(\dfrac{|B_{t^{-1}R}|}{|B_R|}\Big)^{\zeta\lambda} \lesssim t^{-Q\zeta\lambda},\,\,\,\,\,\,\,\,\,\,\,\,\,\,\,\,\,\,\,\,\,\,\,\,\,\,\,\,\,\,\,\textit{\rm if}\,\, t\geq 1,
\\
\\
\Big(\dfrac{|B_{t^{-1}R}|}{|B_R|}\Big)^{\lambda(\delta-1)/\delta} \lesssim t^{-Q\lambda(\delta-1)/\delta},\,\textit{\rm otherwise}.
\end{array} \right.
\end{align}
Consequently, by (\ref{HfMor3}), we have
\begin{align}
\|\mathcal H_{\Phi,\Omega}(f)\|_{\mathop{B}_{\omega}^{q,\lambda}(\mathbb H^n)}&\lesssim \Big(\int_{1}^{\infty}\dfrac{\Phi(t)}{t^{1+Q\zeta\lambda}}dt+\int_{0}^{1}\dfrac{\Phi(t)}{t^{1+Q\lambda(\delta-1)/\delta}}dt\Big)\|\Omega\|_{L^{(\frac{q^*}{\zeta})'}(S_{Q-1})}\|f\|_{\mathop{B}_{\omega}^{q^*,\lambda}(\mathbb H^n)},\nonumber
\end{align}
which ends the proof of this theorem.
\end{proof}
Next, we also give the necessary and sufficient conditions for the boundedness of the rough Hausdorff operator on the weighted Herz spaces on the Heisenberg group.
\begin{theorem}\label{TheoremHerz}
Let  $1\leq p<\infty$, $1< q<\infty$, $\gamma\in\mathbb R$, $\omega(x)= |x|^\gamma_h$ and $\Omega\in L^{q'}(S_{Q-1})$. Then, ${\mathcal H}_{\Phi,\Omega}$ is bounded from ${K}^{\alpha,p}_{q,\omega}(\mathbb H^n)$ to itself  if and only if 
$$
\mathcal C_{3}= \int_{0}^{\infty}{\dfrac{\Phi(t)}{t^{1-\alpha-\frac{Q+\gamma}{q}}}}dt <  + \infty.
$$
Moreover, ${\big\|{\mathcal H}_{\Phi,\Omega}\big\|_{{K}^{\alpha,p}_{q,\omega}(\mathbb H^n)\to {K}^{\alpha,p}_{q, \omega}(\mathbb H^n)}\simeq \mathcal C_{3}}.$
\end{theorem}
\begin{proof}
To prove the sufficient condition of theorem, let us assume that $\mathcal C_3<\infty$. By estimating as (\ref{HfMorrey0}) above, it follows that
\begin{align}
\|{\mathcal H}_{\Phi,\Omega}(f)\chi_k\|_{L^q_\omega(\mathbb H^n)}\leq \|\Omega\|_{L^{q'}(S_{Q-1})}\int_{0}^{\infty}\dfrac{\Phi(t)}{t}\Big(\int_{S_{Q-1}}\int_{C_k}|f(\delta_{t^{-1}|x|_h}y')|^q\omega(x)dxdy'\Big)^{1/q}dt\nonumber.
\end{align}
In addition, a similar argument as (\ref{S0BR}) above, we imply
\begin{align}
&\int_{S_{Q-1}}\int_{C_k}|f(\delta_{t^{-1}|x|_h}y')|^q\omega(x)dxdy'\simeq t^{Q+\gamma}\int_{\delta_{t^{-1}}C_k}|f(x)|^q \omega(x)dx,\nonumber
\end{align}
where $\delta_{t^{-1}}C_k=\{z\in\mathbb H^n: \frac{2^{k-1}}{t}<|z|_h\leq \frac{2^k}{t}\}$. 
This leads to
\begin{align}\label{HfHerzchik}
\|{\mathcal H}_{\Phi,\Omega}(f)\chi_k\|_{L^q_\omega(\mathbb H^n)}\leq \|\Omega\|_{L^{q'}(S_{Q-1})}\int_{0}^{\infty}\dfrac{\Phi(t)}{t^{1
-\frac{Q+\gamma}{q}}}\|f\chi_{\delta_{t^{-1}}C_k}\|_{L^q_\omega(\mathbb H^n)}dt.
\end{align}
Note that for any $t>0$, there exists an integer number $\ell=\ell(t)$ satisfying $2^{\ell-1}< t^{-1}\leq 2^{\ell}$. It is clear that $\delta_{t^{-1}}C_k\subset\{z\in \mathbb H^n: 2^{k+\ell-2}<|z|_h\leq 2^{k+\ell}\}.$ 
Hence, by (\ref{HfHerzchik}), it is easy to see that
\begin{align}\label{HfMHerzchik}
\|{\mathcal H}_{\Phi,\Omega}(f)\chi_k\|_{L^q_\omega(\mathbb H^n)}\leq \|\Omega\|_{L^{q'}(S_{Q-1})}\int_{0}^{\infty}\dfrac{\Phi(t)}{t^{1-\frac{Q+\gamma}{q}}}\Big(\|f\chi_{k+\ell-1}\|_{L^q_\omega(\mathbb H^n)}+\|f\chi_{k+\ell}\|_{L^q_\omega(\mathbb H^n)}\Big)dt,
\end{align}
for all $k\in\mathbb Z$. Consequently, by the Minkowski inequaltiy, we have
$\|{\mathcal H}_{\Phi,\Omega}(f)\|_{K^{\alpha,p}_{q,\omega}(\mathbb H^n)}$ is dominated by
\begin{align}
&\leq \|\Omega\|_{L^{q'}(S_{Q-1})}\Big(\sum\limits_{k=-\infty}^{\infty}\Big(2^{k\alpha}\int_{0}^{\infty}\dfrac{\Phi(t)}{t^{1-\frac{Q+\gamma}{q}}}(\|f\chi_{k+\ell-1}\|_{L^q_\omega(\mathbb H^n)}+\|f\chi_{k+\ell}\|_{L^q_\omega(\mathbb H^n)})\Big)^p\Big)^{1/p}
\nonumber
\\
&\lesssim \|\Omega\|_{L^{q'}(S_{Q-1})}\int_{0}^{\infty}\dfrac{\Phi(t)}{t^{1-\frac{Q+\gamma}{q}}}\Big\{\Big(\sum\limits_{k=-\infty}^{\infty}2^{k\alpha p}\|f\chi_{k+\ell-1}\|_{L^q_\omega(\mathbb H^n)}^p\Big)^{1/p}\nonumber
\\
&\,\,\,\,\,\,\,\,\,\,\,\,\,\,\,\,\,\,\,\,\,\,\,\,\,\,\,\,\,\,\,\,\,\,\,\,\,\,\,\,\,\,\,\,\,\,\,\,\,\,\,\,\,\,\,\,\,\,\,\,\,\,\,\,\,\,\,\,\,\,\,\,\,\,\,\,\,\,\,\,\,\,\,\,\,\,\,\,\,\,\,\,\,\,\,\,\,\,\,\,\,\,\,\,\,+\Big(\sum\limits_{k=-\infty}^{\infty}2^{k\alpha p}\|f\chi_{k+\ell-1}\|_{L^q_\omega(\mathbb H^n)}^p\Big)^{1/p}\Big\}dt.
\nonumber
\end{align}
Since $2^{\ell-1}<t^{-1}\leq 2^{\ell}$ and the definition of the Herz space, we estimate
\begin{align}
&\Big(\sum\limits_{k=-\infty}^{\infty}2^{k\alpha p}\|f\chi_{k+\ell-1}\|_{L^q_\omega(\mathbb H^n)}^p\Big)^{1/p}+\Big(\sum\limits_{k=-\infty}^{\infty}2^{k\alpha p}\|f\chi_{k+\ell-1}\|_{L^q_\omega(\mathbb H^n)}^p\Big)^{1/p}\nonumber
\\
&\leq \Big(2^{(1-\ell)\alpha}+2^{-\ell\alpha}\Big)\|f\|_{K^{\alpha,p}_{q,\omega}(\mathbb H^n)}\lesssim t^{\alpha}.\|f\|_{K^{\alpha,p}_{q,\omega}(\mathbb H^n)}\nonumber.
\end{align}
Therefore, we obtain
\begin{align}
\|{\mathcal H}_{\Phi,\Omega}(f)\|_{K^{\alpha,p}_{q,\omega}(\mathbb H^n)}\lesssim \mathcal C_3.\|\Omega\|_{L^{q'}(S_{Q-1})}\|f\|_{K^{\alpha,p}_{q,\omega}(\mathbb H^n)},\nonumber
\end{align}
 which implies that ${\mathcal H}_{\Phi,\Omega}$ is bounded from ${K}^{\alpha,p}_{q,\omega}(\mathbb H^n)$ to itself.
\vskip 5pt
Conversely, suppose that ${\mathcal H}_{\Phi,\Omega}$ is bounded on ${K}^{\alpha,p}_{q,\omega}(\mathbb H^n)$. Then, we also choose the function $f$ as follows
$$
f(x) = \left\{ \begin{array}{l}
0,\,\,\,\,\,\,\,\,\,\,\,\,\,\,\,\,\,\,\,\,\,\,\,\,\,\,\,\,\,\,\,\,\,\,\,\,\,\,\,\,\,\,\,\,\,\,\,\,\,\,\,\,\,\,\,\,\,\,\,\,\,\,\,\,\,\,\,\,\,\,\,\,\,\,\,\,\,\,\,\,\,\,\,\,\,\,\,\,\,\,\,\,\,\,\textit{\rm if }\,\,|x|_h<1,
\\
|x|_h^{-\alpha-\frac{Q+\gamma}{q}-\varepsilon}.|\Omega(\delta_{|x|_h^{-1}}x)|^{q'-2}.\overline{\Omega}(\delta_{|x|_h^{-1}}x),\,\textit{\rm otherwise.}
\end{array} \right.
$$
It is evident that $\|f\chi_k\|_{L^q_\omega(\mathbb H^n)}=0$ for every $k<0$. Otherwise, for $k\geq 0$  we have 
\begin{align}\label{fCkHerz}
\|f\chi_k\|_{L^q_\omega(\mathbb H^n)}=&\Big(\int_{C_k}|x|_h^{-\alpha q-(Q+\gamma)-\varepsilon q}|\Omega(\delta_{|x|_h^{-1}}x)|^{q'}\omega(x)dx\Big)^{1/q}\nonumber
\\
&=\Big(\int_{2^{k-1}}^{2^k}\int_{S_{Q-1}}|\delta_ry'|_h^{-\alpha q-(Q+\gamma)-\varepsilon q+\gamma}|\Omega(y')|^{q'}r^{Q-1}dy'dr\Big)^{1/q}\nonumber
\\
&=\Big(\int_{2^{k-1}}^{2^k}\int_{S_{Q-1}}r^{-(\alpha+\varepsilon)q-1}|\Omega(y')|^{q'}dy'dr\Big)^{1/q}\nonumber\\
&=2^{-(\alpha+\varepsilon)k}\Big(\frac{2^{q(\alpha+\varepsilon)-1}}{q(\alpha+\varepsilon)}\Big)^{1/q}\|\Omega\|_{L^q(S_{Q-1})}^{q'/q}.
\end{align}
Thus, by the definition of the  space ${K}^{\alpha,p}_{q,\omega}(\mathbb H^n)$, we have
\begin{align}\label{fHerznorm}
\|f\|_{K^{\alpha,p}_{q,\omega}(\mathbb H^n)}&=\Big(\sum\limits_{k=1}^{\infty}2^{k\alpha p}\Big(2^{-(\alpha+\varepsilon)k}\Big(\frac{2^{q(\alpha+\varepsilon)-1}}{q(\alpha+\varepsilon)}\Big)^{1/q}\|\Omega\|_{L^q(S_{Q-1})}^{q'/q}\Big)^p\Big)^{1/p}
\nonumber
\\
&=\Big(\frac{2^{q(\alpha+\varepsilon)-1}}{q(\alpha+\varepsilon)}\Big)^{1/q}\|\Omega\|_{L^q(S_{Q-1})}^{q'/q}\Big(\sum\limits_{k=1}^{\infty}2^{-k\varepsilon p}\Big)^{1/p}\nonumber
\\
&=\Big(\frac{2^{q(\alpha+\varepsilon)-1}}{q(\alpha+\varepsilon)}\Big)^{1/q}\|\Omega\|_{L^q(S_{Q-1})}^{q'/q}\frac{1}{(2^{\varepsilon p}-1)^{1/p}}<\infty.
\end{align}
Since $\delta_{|\delta_{t|x|_h}y'|^{-1}_h}\delta_{t|x|_h}y'=y'$, we obtain
\begin{align}
{\mathcal H}_{\Phi,\Omega}(f)(x)&=\int_{0}^{\infty}\int_{S_{Q-1}}\dfrac{\Phi(\frac{1}{t})}{t}\Omega(y')f(\delta_{t|x|_h}y')dy'dt\nonumber
\\
&=\int_{\frac{1}{|x|_h}}^{\infty}\int_{S_{Q-1}}\dfrac{\Phi(\frac{1}{t})}{t}\Omega(y')
|\delta_{t|x|_h}y'|^{-\alpha-\frac{Q+\gamma}{q}-\varepsilon}
|\Omega(y')|^{q'-2}\overline{\Omega}(y')dy'dt.\nonumber
\\
&=\|\Omega\|^{q'}_{L^{q'}(S_{Q-1})}.\Big(\int_{\frac{1}{|x|_h}}^{\infty}\dfrac{\Phi(\frac{1}{t})}{t^{1+\alpha+\frac{Q+\gamma}{q}+\varepsilon}}dt\Big).|x|_h^{-\alpha-\frac{Q+\gamma}{q}-\varepsilon}.\nonumber
\end{align}
Hence, by setting up $g_{\varepsilon}(x)=|x|_h^{-\alpha-\frac{Q+\gamma}{q}-\varepsilon}\chi_{\mathbb H^n\setminus B(0,\varepsilon^{-1})}$, we deduce
\begin{align}\label{HfHerz}
{\mathcal H}_{\Phi,\Omega}(f)(x)&\geq\|\Omega\|^{q'}_{L^{q'}(S_{Q-1})}.\Big(\int_{0}^{\infty}\dfrac{\Phi(\frac{1}{t})\chi_{(\varepsilon,\infty)}(t)}{t^{1+\alpha+\frac{Q+\gamma}{q}+\varepsilon}}dt\Big). g_{\varepsilon}(x),
\end{align}
We denote by $k_\varepsilon$ the integer number such that $2^{k_\varepsilon-1}\geq \varepsilon^{-1}>2^{k_\varepsilon-2}$. As a consequence, by making as (\ref{fCkHerz}) above, it follows that
\begin{align}
\|g_{\varepsilon}\|_{K^{\alpha,p}_{q,\omega}(\mathbb H^n)}&\geq \Big(\sum\limits_{k=k_{\varepsilon}}^{\infty}2^{k\alpha p}\|g_{\varepsilon}
\chi_k\|_{L^q_\omega(\mathbb H^n)}^p\Big)^{1/p}= \Big(\sum\limits_{k=k_{\varepsilon}}^{\infty}2^{k\alpha p}\Big(2^{-(\alpha+\varepsilon)k}\Big(\frac{2^{q(\alpha+\varepsilon)-1}}{q(\alpha+\varepsilon)}\Big)^{1/q}\Big)^p\Big)^{1/p}
\nonumber
\\
&=\Big(\frac{2^{q(\alpha+\varepsilon)-1}}{q(\alpha+\varepsilon)}\Big)^{1/q}\Big(\sum\limits_{k=k_{\varepsilon}}^{\infty}2^{-k\varepsilon p}\Big)^{1/p}=\Big(\frac{2^{q(\alpha+\varepsilon)-1}}{q(\alpha+\varepsilon)}\Big)^{1/q}\frac{2^{\varepsilon(1-k_{\varepsilon})}}{(2^{\varepsilon p}-1)^{1/p}}.\nonumber
\end{align}
From this reason, by (\ref{fHerznorm}) and (\ref{HfHerz}), it leads to
\begin{align}
\|{\mathcal H}_{\Phi,\Omega}(f)\|_{K^{\alpha,p}_{q,\omega}(\mathbb H^n)}&\geq \|\Omega\|^{q'}_{L^{q'}(S_{Q-1})}.\Big(\int_{0}^{\infty}\dfrac{\Phi(\frac{1}{t})\chi_{(\varepsilon,\infty)}(t)}{t^{1+\alpha+\frac{Q+\gamma}{q}+\varepsilon}}dt\Big).\|g_{\varepsilon}\|_{K^{\alpha,p}_{q,\omega}(\mathbb H^n)}\nonumber
\\
&\geq \|\Omega\|_{L^{q'}(S_{Q-1})}.\Big(\int_{0}^{\infty}\dfrac{\Phi(\frac{1}{t})\chi_{(\varepsilon,\infty)}(t)}{t^{1+\alpha+\frac{Q+\gamma}{q}+\varepsilon}}dt\Big).2^{\varepsilon(1-k_{\varepsilon})}.\|f\|_{K^{\alpha,p}_{q,\omega}(\mathbb H^n)}.\nonumber
\end{align} 
Since  $2^{k_\varepsilon-1}\geq \varepsilon^{-1}>2^{k_\varepsilon-2}$, we imply that $2^{k_\varepsilon}\simeq  \varepsilon^{-1}$. This implies that
\begin{align}
\|{\mathcal H}_{\Phi,\Omega}(f)\|_{K^{\alpha,p}_{q,\omega}(\mathbb H^n)}&\gtrsim 2^{\varepsilon}\varepsilon^{\varepsilon}\|\Omega\|_{L^{q'}(S_{Q-1})}.\Big(\int_{0}^{\infty}\dfrac{\Phi(\frac{1}{t})\chi_{(\varepsilon,\infty)}(t)}{t^{1+\alpha+\frac{Q+\gamma}{q}+\varepsilon}}dt\Big).\|f\|_{K^{\alpha,p}_{q,\omega}(\mathbb H^n)}.\nonumber
\end{align}
Remark that $t^{-\varepsilon}\chi_{(\varepsilon,\infty)}(t)\lesssim 1$ with $\varepsilon$ sufficiently small and $\mathop{\rm lim}\limits_{\varepsilon\to 0^{+}} 2^{\varepsilon}\varepsilon^{\varepsilon}=1$. Thus, by the dominated convergence theorem of Lebesgue, we immediately obtain
$$
\int_{0}^{\infty}\dfrac{\Phi(\frac{1}{t})}{t^{1+\alpha+\frac{Q+\gamma}{q}}}dt<\infty.
$$
By a change of variables, we finish the proof of this theorem.
\end{proof}
For the Herz spaces associated with the Muckenhoupt weights, we also have the boundedness of the rough Hausdorff operators on the Heisenberg group. More precisely, we have the following interesting results.
\begin{theorem}\label{TheoremHerz1}
Let $1\leq p, q^*, q,\zeta<\infty$, $\alpha\in\mathbb R$, $\alpha^*<0$ satisfying
\begin{align}\label{alpha*alpha}
\dfrac{1}{q^*}+\dfrac{\alpha^*}{Q}=\dfrac{1}{q}+\dfrac{\alpha}{Q}.
\end{align}
Suppose that $\omega\in A_{\zeta}$ with the finite critical index $r_{\omega}$ for the reverse H\"{o}lder condition and $q^*>q\zeta r_{\omega}/(r_{\omega}-1)$, $\Omega\in L^{(\frac{q^*}{\zeta})'}(S_{Q-1}),\delta\in (1,r_\omega)$.
\\

$(\rm i)$ If $\dfrac{1}{q^*}+\dfrac{\alpha^*}{Q}\geq 0$ and 
$$\mathcal C_{4.1}= \int_{0}^{1/2}\dfrac{\Phi(t)}{t^{1-Q\frac{(\delta-1)}{\delta}(\frac{1}{q^*}+\frac{\alpha}{Q^*})}}dt +\int_{1/2}^{\infty}\dfrac{\Phi(t)}{t^{1-Q\zeta(\frac{1}{q^*}+\frac{\alpha^*}{Q})}}dt<\infty,
$$
then ${\mathcal H}_{\Phi,\Omega}$ is bounded from ${\mathop{K}\limits^.}^{\alpha^*,p}_{q^*,\omega}(\mathbb H^n)$ to ${\mathop{K}\limits^.}^{\alpha, p}_{q,\omega}(\mathbb H^n)$.
\\

$\,(\rm ii)$ If $\dfrac{1}{q^*}+\dfrac{\alpha^*}{Q}<0$ and 
$$\mathcal C_{4.2}= \int_{0}^{1/2}\dfrac{\Phi(t)}{t^{1-Q\zeta(\frac{1}{q^*}+\frac{\alpha^*}{Q})}}dt+\int_{1/2}^{\infty}\dfrac{\Phi(t)}{t^{1-Q\frac{(\delta-1)}{\delta}(\frac{1}{q^*}+\frac{\alpha^*}{Q})}}dt<\infty,
$$
then ${\mathcal H}_{\Phi,\Omega}$ is bounded from ${\mathop{K}\limits^.}^{\alpha^*, p}_{q^*,\omega}(\mathbb H^n)$ to ${\mathop{K}\limits^.}^{\alpha, p}_{q,\omega}(\mathbb H^n)$.
\end{theorem}
\begin{proof}
From Definition \ref{RH1} and the similar reason as (\ref{HfBRMorrey}), for $k\in\mathbb Z$, we estimate
\begin{align}
\|{\mathcal H}_{\Phi,\Omega}(f)\chi_k\|_{L^{q}_\omega(\mathbb H^n)}\lesssim \|\Omega\|_{L^{(\frac{q^*}{\zeta})'}(S_{Q-1})}\int_{0}^{\infty}\dfrac{\Phi(t)}{t}\dfrac{\omega(U_k)^{\frac{1}{q}}}{\omega(\delta_{t^{-1}}U_k)^{\frac{1}{q^*}}}\|f\|_{L^{q^*}_\omega(\delta_{t^{-1}}U_k)}dt,\nonumber
\end{align}
where $\delta_{t^{-1}}U_k=\{z\in\mathbb H^n: |z|_h\leq \frac{2^{k}}{t}\}.$
Hence, by making the Minkowski inequality and the relation (\ref{alpha*alpha}), $\|{\mathcal H}_{\Phi,\Omega}(f)\|_{{\mathop{K}\limits^{.}}^{\alpha,p}_{q,\omega}(\mathbb H^n)}$ is dominated by
\begin{align}
&\|\Omega\|_{L^{(\frac{q^*}{\zeta})'}(S_{Q-1})}\Big(\sum\limits_{k=-\infty}^{\infty}\omega(U_k)^{\frac{\alpha p}{Q}}\Big(\int_{0}^{\infty}\dfrac{\Phi(t)}{t}\dfrac{\omega(U_k)^{\frac{1}{q}}}{\omega(\delta_{t^{-1}}U_k)^{\frac{1}{q^*}}}\|f\|_{L^{q^*}_\omega(\delta_{t^{-1}}U_k)}dt\Big)^p\Big)^{1/p}.\nonumber
\\
&\leq \|\Omega\|_{L^{(\frac{q^*}{\zeta})'}(S_{Q-1})}\int\limits_{0}^{\infty}\dfrac{\Phi(t)}{t}\Big(\sum\limits_{k=-\infty}^{\infty}\Big(\dfrac{\omega(U_k)^{\frac{1}{q}+\frac{\alpha}{Q}}}{\omega(\delta_{t^{-1}}U_k)^{\frac{1}{q^*}}}\|f\|_{L^{q^*}_\omega(\delta_{t^{-1}}U_k)}\Big)^p\Big)^{\frac{1}{p}}dt.\nonumber
\\
&= \|\Omega\|_{L^{(\frac{q^*}{\zeta})'}(S_{Q-1})}\sum\limits_{j\in\mathbb Z}\,\int_{t^{-1}\in (2^{j},2^{j+1}]}\dfrac{\Phi(t)}{t}
\Big(\sum\limits_{k=-\infty}^{\infty}\Big(\dfrac{\omega(U_k)^{\frac{1}{q^*}+\frac{\alpha^*}{Q}}}{\omega(\delta_{t^{-1}}U_k)^{\frac{1}{q^*}}}\|f\|_{L^{q^*}_\omega(\delta_{t^{-1}}U_k)}\Big)^p\Big)^{\frac{1}{p}}dt.\nonumber
\end{align}
For $t^{-1}\in (2^{j},2^{j+1}]$, we have $U_{k+j}\subset\delta_{t^{-1}}U_k\subset U_{k+j+1}$. Hence, it follows that 
\begin{align}
\|{\mathcal H}_{\Phi,\Omega}(f)\|_{{\mathop{K}\limits^{.}}^{\alpha,p}_{q,\omega}(\mathbb H^n)}&\lesssim \|\Omega\|_{L^{(\frac{q^*}{\zeta})'}(S_{Q-1})}\sum\limits_{j\in\mathbb Z}\,\int_{t^{-1}\in (2^{j},2^{j+1}]}\dfrac{\Phi(t)}{t}\times
\nonumber
\\
&\,\,\,\,\,\,\,\,\,\,\,\,\,\,\,\times\Big(\sum\limits_{k=-\infty}^{\infty}\Big(\dfrac{\omega(U_k)^{\frac{1}{q^*}+\frac{\alpha^*}{Q}}}{\omega(U_{k+j})^{\frac{1}{q^*}}}\|f\|_{L^{q^*}_\omega(U_{k+j+1})}\Big)^p\Big)^{\frac{1}{p}}dt.\nonumber
\end{align}
Then,
\begin{align}\label{HfHerzAp1}
&\|{\mathcal H}_{\Phi,\Omega}(f)\|_{{\mathop{K}\limits^{.}}^{\alpha,p}_{q,\omega}(\mathbb H^n)}\lesssim\|\Omega\|_{L^{(\frac{q^*}{\zeta})'}(S_{Q-1})}\sum\limits_{j\in\mathbb Z}\,\int_{ t^{-1}\in (2^{j},2^{j+1}]}\dfrac{\Phi(t)}{t}
\times\nonumber
\\
&\,\,\,\times\Big(\sum\limits_{k=-\infty}^{\infty}\Big\{\Big(\dfrac{\omega(U_k)}{\omega(U_{k+j})}\Big)^{\frac{1}{q^*}+\frac{\alpha^*}{Q}}\sum\limits_{\ell=-\infty}^{j+1}\Big(\dfrac{\omega(U_{k+j})}{\omega(U_{k+\ell})}\Big)^{\frac{\alpha^*}{Q}}\omega(U_{k+\ell})^{\frac{\alpha^*}{Q}}\|f\chi_{k+\ell}\|_{L^{q^*}_\omega(\mathbb H^n)}\Big\}^p\Big)^{1/p}dt.
\end{align}
For simplicity, we denote
\[
\mathcal A_{j,\ell}= \left\{ \begin{array}{l}
2^{(j-\ell)\alpha^*(\delta-1)/\delta},\,\textit{\rm if }\, \ell\leq j,
\\
 2^{-\zeta\alpha^*},\,\,\,\,\,\,\,\,\,\,\,\,\,\,\,\,\,\,\,\,\,\textit{\rm if }\, \ell = j+1.
\end{array} \right.\]
By Proposition \ref{pro2.3DFan} with $\alpha^*<0$ and $\ell\leq j+1$, we obtain
\begin{align}\label{omegaHez11}
\Big(\dfrac{\omega(U_{k+j})}{\omega(U_{k+\ell})}\Big)^{\frac{\alpha^*}{Q}}\lesssim \left\{ \begin{array}{l}
\Big(\dfrac{|U_{k+j}|}{|U_{k+\ell}|}\Big)^{\frac{\alpha^*(\delta-1)}{Q\delta}},\,\textit{\rm if }\, \ell\leq j,
\\
\\
\Big(\dfrac{|U_{k+j}|}{|U_{k+\ell}|}\Big)^{\frac{\zeta\alpha^*}{Q}},\,\,\,\,\,\,\,\,\,\textit{\rm if }\, \ell = j+1.
\end{array} \right.
\lesssim \mathcal A_{j,\ell}.
\end{align}
It is useful to note that by Proposition \ref{pro2.3DFan} we have to consider the following two cases.
\vskip 5pt
Case 1: $\frac{1}{q^*}+\frac{\alpha^*}{Q}\geq 0$ and $t^{-1}\in (2^{j},2^{j+1}]$. We have
\begin{align}\label{omegaHez12}
\Big(\dfrac{\omega(U_k)}{\omega(U_{k+j})}\Big)^{\frac{1}{q^*}+\frac{\alpha^*}{Q}} \lesssim \left\{ \begin{array}{l}
\Big(\dfrac{|U_k|}{|U_{k+j}|}\Big)^{\zeta(\frac{1}{q^*}+\frac{\alpha^*}{Q})}\lesssim 2^{-jQ\zeta(\frac{1}{q^*}+\frac{\alpha^*}{Q})}\lesssim t^{Q\zeta(\frac{1}{q^*}+\frac{\alpha^*}{Q})},\,\,\,\,\,\,\,\,\,\,\,\,\,\,\,\,\,\,\,\,\,\,\,\,\,\,\,\textit{\rm if}\,j\leq 0,
\\
\\
\Big(\dfrac{|U_k|}{|U_{k+j}|}\Big)^{\frac{(\delta-1)}{\delta}(\frac{1}{q^*}+\frac{\alpha^*}{Q})}\lesssim 2^{-jQ\frac{(\delta-1)}{\delta}(\frac{1}{q^*}+\frac{\alpha^*}{Q})}\lesssim t^{Q\frac{(\delta-1)}{\delta}(\frac{1}{q^*}+\frac{\alpha^*}{Q})},\,\textit{\rm otherwise}.
\end{array} \right.
\end{align}
\vskip 5pt
Case 2: $\frac{1}{q^*}+\frac{\alpha^*}{Q}<0$ and $t^{-1}\in (2^{j},2^{j+1}]$. Then,  we also have
\begin{align}\label{omegaHez13}
\Big(\dfrac{\omega(U_k)}{\omega(U_{k+j})}\Big)^{\frac{1}{q^*}+\frac{\alpha^*}{Q}} \lesssim \left\{ \begin{array}{l}
\Big(\dfrac{|U_k|}{|U_{k+j}|}\Big)^{\zeta(\frac{1}{q^*}+\frac{\alpha^*}{Q})}\lesssim 2^{-jQ\zeta(\frac{1}{q^*}+\frac{\alpha^*}{Q})}\lesssim t^{\zeta Q(\frac{1}{q^*}+\frac{\alpha^*}{Q})},\,\,\,\,\,\,\,\,\,\,\,\,\,\,\,\,\,\,\,\,\,\,\,\,\,\,\textit{\rm if}\,j> 0,
\\
\\
\Big(\dfrac{|U_k|}{|U_{k+j}|}\Big)^{\frac{(\delta-1)}{\delta}(\frac{1}{q^*}+\frac{\alpha^*}{Q})}\lesssim 2^{-jQ\frac{(\delta-1)}{\delta}(\frac{1}{q^*}+\frac{\alpha^*}{Q})}\lesssim t^{Q\frac{(\delta-1)}{\delta}(\frac{1}{q^*}+\frac{\alpha^*}{Q})},\,\textit{\rm otherwise}.
\end{array} \right.
\end{align}
Now, by (\ref{HfHerzAp1}), (\ref{omegaHez11}) and (\ref{omegaHez12}), we have
\begin{align}
&\|{\mathcal H}_{\Phi,\Omega}(f)\|_{{\mathop{K}\limits^{.}}^{\alpha,p}_{q,\omega}(\mathbb H^n)}\lesssim \|\Omega\|_{L^{(\frac{q^*}{\zeta})'}(S_{Q-1})}\times\nonumber
\\
&\times\Big(\sum\limits_{j=-\infty}^{0}\,\int_{ t^{-1}\in (2^{j},2^{j+1}]}\dfrac{\Phi(t)}{t^{1-Q\zeta(\frac{1}{q^*}+\frac{\alpha^*}{Q})}}F_jdt +\sum\limits_{j=1}^{\infty}\,\int_{ t^{-1}\in (2^{j},2^{j+1}]}\dfrac{\Phi(t)}{t^{1-Q\frac{(\delta-1)}{\delta}(\frac{1}{q^*}+\frac{\alpha^*}{Q})}}F_jdt\Big),\nonumber
\end{align}
where $F_j=\Big(\sum\limits_{k=-\infty}^{\infty}\Big\{\sum\limits_{\ell=-\infty}^{j+1}\mathcal A_{j,\ell}.\omega(B_{k+\ell})^{\frac{\alpha^*}{Q}}\|f\chi_{k+\ell}\|_{L^{q^*}_\omega(\mathbb H^n)}\Big\}^p\Big)^{1/p}.$
By the Minkowski inequality and the definition of $\mathcal A_{j,\ell}$ above, we get
\begin{align}
F_j&\leq \sum\limits_{\ell=-\infty}^{j+1}\mathcal A_{j,\ell}\Big\{\sum\limits_{k=-\infty}^{\infty}\Big(\omega(B_{k+\ell})^{\frac{\alpha^*}{Q}}\|f\chi_{k+\ell}\|_{L^{q^*}_\omega(\mathbb H^n)}\Big)^p\Big\}^{1/p}\nonumber
\\
 &=\Big(\sum\limits_{\ell=-\infty}^{j}2^{(j-\ell)\alpha^*(\delta-1)/\delta}+2^{-\zeta\alpha^*}\Big)\|f\|_{{\mathop{K}\limits^{.}}^{\alpha^*, p}_{q^*,\omega}(\mathbb H^n)}\lesssim \|f\|_{{\mathop{K}\limits^{.}}^{\alpha^*, p}_{q^*,\omega}(\mathbb H^n)}.\nonumber
\end{align}
Therefore, we have
\begin{align}
\|{\mathcal H}_{\Phi,\Omega}(f)\|_{{\mathop{K}\limits^{.}}^{\alpha,p}_{q,\omega}(\mathbb H^n)}&\lesssim \|\Omega\|_{L^{(\frac{q^*}{\zeta})'}(S_{Q-1})}\Big(\sum\limits_{j=-\infty}^{0}\,\int_{ t^{-1}\in (2^{j},2^{j+1}]}\dfrac{\Phi(t)}{t^{1-Q\zeta(\frac{1}{q^*}+\frac{\alpha^*}{Q})}}dt +\nonumber
\\
&+\sum\limits_{j=1}^{\infty}\,\int_{ t^{-1}\in (2^{j},2^{j+1}]}\dfrac{\Phi(t)}{t^{1-Q\frac{(\delta-1)}{\delta}(\frac{1}{q^*}+\frac{\alpha^*}{Q})}}dt\Big)\|f\|_{{\mathop{K}\limits^{.}}^{\alpha^*, p}_{q^*,\omega}(\mathbb H^n)}\nonumber
\\
&=\mathcal C_{4.1}.\|\Omega\|_{L^{(\frac{q^*}{\zeta})'}(S_{Q-1})}.\|f\|_{{\mathop{K}\limits^{.}}^{\alpha^*, p}_{q^*,\omega}(\mathbb H^n)},\nonumber
\end{align}
which finishes the proof for part (i).
\vskip 5pt
Next, we will prove for part (ii). By combining (\ref{HfHerzAp1}), (\ref{omegaHez11}), (\ref{omegaHez13}) and making a similar estimation as above, it is not difficult to obtain that
\begin{align}
\|{\mathcal H}_{\Phi,\Omega}(f)\|_{{\mathop{K}\limits^{.}}^{\alpha,p}_{q,\omega}(\mathbb H^n)}\lesssim \mathcal C_{4.2}.\|\Omega\|_{L^{(\frac{q^*}{\zeta})'}(S_{Q-1})}.\|f\|_{{\mathop{K}\limits^{.}}^{\alpha^*, p}_{q^*,\omega}(\mathbb H^n)}.\nonumber
\end{align}
Therefore, the proof of the theorem is achieved.
\end{proof}
By using the block decomposition of the weighted Herz space as mentioned in Section \ref{2}, we also obtain the boundedness of rough Hausdorff  operator on the weighted Herz space for the case $0<p<1$. More details, we have the following result.
\begin{theorem}\label{TheoremHerz2}
Let $0<p<1\leq q^*, q <\infty$, $\alpha^* >0,\alpha>0$ and $\sigma>(1-p)/p$ such that
$$\dfrac{1}{q^*}+\dfrac{\alpha^*}{Q}=\dfrac{1}{q}+\dfrac{\alpha}{Q}.
$$
At the same time, let $\omega\in A_{1}$ with the finite critical index $r_{\omega}$ for the reverse H\"{o}lder condition and $q^*>q r_{\omega}/(r_{\omega}-1)$, $\Omega\in L^{{q^*}'}(S_{Q-1}),\delta\in (1,r_\omega)$ and the following condition is true:
$$
\mathcal C_5=\int_{0}^{1/2}\dfrac{\Phi(t)}{t^{1-\frac{Q(\delta-1)}{\delta}(\frac{1}{q^*}+\frac{\alpha^*}{Q})}}|{\rm log}_2(t)|^{\sigma}dt + \int_{1/2}^{\infty}\dfrac{\Phi(t)}{t^{1-\frac{Q}{q^*}-\alpha^*}}({\rm log}_2(t)+1)^{\sigma}dt<\infty.
$$
Then, ${\mathcal H}_{\Phi,\Omega}$ is bounded from ${\mathop{K}\limits^.}^{\alpha^*,p}_{q^*,\omega}(\mathbb H^n)$ to ${\mathop{K}\limits^.}^{\alpha, p}_{q,\omega}(\mathbb H^n)$.
\end{theorem}
\begin{proof}
Let $f\in {\mathop{K}\limits^.}^{\alpha^*,p}_{q^*,\omega}(\mathbb H^n)$. By Proposition \ref{blockHerz}, we have a block decomposition for $f$ as follows

$$f=\sum\limits_{k=-\infty}^{\infty}\lambda_k.b_k,$$
where $\textit{\rm and}\, \Big(\sum\limits_{k=-\infty}^{\infty}|\lambda_k|^p\Big)^{1/p}\lesssim \|f\|_{{\mathop{K}\limits^.}^{\alpha^*,p}_{q^*,\omega}(\mathbb H^n)}.$
Note that, for $k\in\mathbb Z$, $b_k$ is central $(\alpha^*,q^*,\omega)$-block satisfying
$${\rm supp} (b_k)\subset U_k\,\,\textit{\rm and}\,\,\|b_k\|_{L^{q^*}_\omega(\mathbb H^n)}\leq \omega(U_k)^{\frac{-\alpha^*}{Q}}.
$$
For convenience, we set
$${\mathcal B}_{\Phi,\Omega}(b_k)(x)=\int_{0}^{\infty}\int_{S_{Q-1}} \dfrac{\Phi(t)}{t}|\Omega(y')|.|b_k(\delta_{t^{-1}|x|_h}y')|dy'dt.
$$
Then, we estimate $|{\mathcal H}_{\Phi,\Omega}(f)(x)|\leq \sum\limits_{k=-\infty}^{\infty}|\lambda_k|.{\mathcal B}_{\Phi,\Omega}(b_k)(x).$
\vskip 5pt
Now, it is necessary  to prove that 
\begin{align}\label{BomgeaK}
\|{\mathcal B}_{\Phi,\Omega}(b_k)\|_{{\mathop{K}\limits^.}^{\alpha,p}_{q,\omega}(\mathbb H^n)}\lesssim \mathcal C_5,\,\,\textit{\rm for all}\,k\in\mathbb Z,
\end{align}
which leads the desired result, that is, 
\begin{align}
\|{\mathcal H}_{\Phi,\Omega}(f)\|_{{\mathop{K}\limits^.}^{\alpha,p}_{q,\omega}(\mathbb H^n)}&\leq \Big(\sum\limits_{k=-\infty}^{\infty}|\lambda_k|^p.\|{\mathcal B}_{\Phi,\Omega}(b_k)\|_{{\mathop{K}\limits^.}^{\alpha,p}_{q,\omega}(\mathbb H^n)}^p\Big)^{\frac{1}{p}}\nonumber
\\
&\lesssim \mathcal C_5\Big(\sum\limits_{k=-\infty}^{\infty}|\lambda_k|^p\Big)^{1/p}\nonumber\\
&\lesssim \mathcal C_5\|f\|_{{\mathop{K}\limits^.}^{\alpha^*,p}_{q^*,\omega}(\mathbb H^n)}.\nonumber
\end{align}
To begin with the estimation (\ref{BomgeaK}), we set
\begin{align}\label{sumgkj}
{\mathcal B}_{\Phi,\Omega}(b_k)(x)&=\sum\limits_{j\in\mathbb Z}\int_{(2^{j-1}, 2^j]}\int_{S_{Q-1}} \dfrac{\Phi(t)}{t}|\Omega(y')|.|b_k(\delta_{t^{-1}|x|_h}y')|dy'dt\nonumber\\
&:=\sum\limits_{j\in\mathbb Z} g_{kj}(x).
\end{align}
Since ${\rm supp}(b_k)\subset U_k$, it is not hard to check that 
\begin{align}\label{suppgkj}
{\rm supp}(g_{kj})\subset U_{k+j}.
\end{align}
Thus, by the Minkowski inequality, we get
$$
\|g_{kj}\|_{L^{q}_\omega(\mathbb H^n)}\leq \int_{(2^{j-1},2^j]}\int_{S_{Q-1}} \dfrac{\Phi(t)}{t}|\Omega(y')|\Big(\int_{U_{k+j}}|b_k(\delta_{t^{-1}|x|_h}y')|^{q}\omega(x) dx\Big)^{1/q}dy'dt.
$$
By making as (\ref{HfBRMorrey}) above and  ${\rm supp}(b_k)\subset U_k$, we obtain
\begin{align}
\|g_{kj}\|_{L^{q}_\omega(\mathbb H^n)}&\lesssim \|\Omega\|_{L^{{q^*}'}(S_{Q-1})}\int_{(2^{j-1}, 2^j]}\dfrac{\Phi(t)}{t}\dfrac{\omega(U_{k+j})^{\frac{1}{q}}}{\omega(\delta_{t^{-1}}U_{k+j})^{\frac{1}{q^*}}}\|b_k\|_{L_\omega^{q^*}(\delta_{t^{-1}}U_{k+j})}dt.\nonumber
\\
&\leq\|\Omega\|_{L^{{q^*}'}(S_{Q-1})}\int_{(2^{j-1}, 2^j]}\dfrac{\Phi(t)}{t}\dfrac{\omega(U_{k+j})^{\frac{1}{q}}}{\omega(U_k)^{\frac{1}{q^*}}}\|b_k\|_{L_\omega^{q^*}(U_k)}dt.
\end{align}
By having the inequality $\|b_k\|_{L^{q^*}_\omega(U_k)}\leq \omega(U_k)^{\frac{-\alpha^*}{Q}}$ and the relation $\frac{1}{q*}+\frac{\alpha^*}{Q}=\frac{1}{q}+\frac{\alpha}{Q}$, we infer
\begin{align}\label{gkjHerz}
\|g_{kj}\|_{L^{q}_\omega(\mathbb H^n)}&\lesssim \|\Omega\|_{L^{{q^*}'}(S_{Q-1})}\Big(\int_{(2^{j-1}, 2^j]}\dfrac{\Phi(t)}{t}dt\Big)\dfrac{\omega(U_{k+j})^{\frac{1}{q}}}{\omega(U_k)^{\frac{1}{q^*}+\frac{\alpha^*}{Q}}}.\nonumber
\\
&\lesssim \|\Omega\|_{L^{{q^*}'}(S_{Q-1})}\Big(\int_{(2^{j-1}, 2^j]}\dfrac{\Phi(t)}{t}dt\Big)\Big(\dfrac{\omega(U_{k+j})}{\omega(U_k)}\Big)^{{\frac{1}{q^*}+\frac{\alpha^*}{Q}}}\omega(U_{k+j})^{\frac{-\alpha}{Q}}.
\end{align}
From Proposition \ref{pro2.3DFan}, it follows that
\begin{align}\label{BkjBk}
\Big(\dfrac{\omega(U_{k+j})}{\omega(U_{k})}\Big)^{\frac{1}{q^*}+\frac{\alpha^*}{Q}} \lesssim \left\{ \begin{array}{l}
\Big(\dfrac{|U_{k+j}|}{|U_{k}|}\Big)^{\frac{1}{q^*}+\frac{\alpha^*}{Q}}\lesssim 2^{jQ(\frac{1}{q^*}+\frac{\alpha^*}{Q})},\,\,\,\,\,\,\,\,\,\,\,\,\,\,\,\,\,\,\,\,\,\,\,\,\,\,\textit{\rm if}\,j\geq 0,
\\
\\
\Big(\dfrac{|U_{k+j}|}{|U_{k}|}\Big)^{\frac{(\delta-1)}{\delta}(\frac{1}{q^*}+\frac{\alpha^*}{Q})}\lesssim 2^{jQ\frac{(\delta-1)}{\delta}(\frac{1}{q^*}+\frac{\alpha^*}{Q})},\,\textit{\rm otherwise}.
\end{array} \right.
\end{align}
For simplicity, we write
\begin{align}
c_{kj}=\left\{ \begin{array}{l}
\|\Omega\|_{L^{{q^*}'}(S_{Q-1})}\Big(\int_{(2^{j-1}, 2^j]}\dfrac{\Phi(t)}{t}dt\Big)2^{jQ(\frac{1}{q^*}+\frac{\alpha^*}{Q})},\,\,\,\,\,\,\,\,\,\,\,\textit{\rm if}\,j\geq 0,
\\
\\
\|\Omega\|_{L^{{q^*}'}(S_{Q-1})}\Big(\int_{(2^{j-1}, 2^j]}\dfrac{\Phi(t)}{t}dt\Big)2^{jQ\frac{(\delta-1)}{\delta}(\frac{1}{q^*}+\frac{\alpha^*}{Q})},\,\textit{\rm otherwise}.\nonumber
\end{array} \right.
\end{align}
Then, by (\ref{gkjHerz}) and (\ref{BkjBk}),  we have
\begin{align}\label{gkjLq*}
\|g_{kj}\|_{L^{q}_\omega(\mathbb H^n)}\lesssim c_{kj}\omega(U_{k+j})^{\frac{-\alpha}{Q}}.
\end{align}
\\
On the other hand, from the definition of $c_{kj}$, we put
\begin{align}
h_{kj}=\left\{ \begin{array}{l}
\dfrac{g_{kj}}{c_{kj}},\,\,\,\,\,\,\,\,\,\,\,\textit{\rm if }\,c_{kj}\neq 0,
\\
\\
g_{kj}=0,\,\,\textit{\rm if }\,c_{kj}= 0.\nonumber
\end{array} \right.
\end{align}
From this, by (\ref{sumgkj}), it is easy to obtain
$$
{\mathcal B}_{\Phi,\Omega}(b_k)(x)=\sum\limits_{j\in\mathbb Z}c_{kj}.h_{kj}(x). 
$$
Moreover, each $h_{kj}$ is a central $(\alpha,q,\omega)$-block. 
Indeed, by (\ref{suppgkj}) and (\ref{gkjLq*}), we have ${\rm supp}(h_{kj})\subset U_{k+j}$ and $\|h_{kj}\|_{L^{q}_\omega(\mathbb H^n)}\lesssim \omega(U_{k+j})^{\frac{-\alpha}{Q}}.$ As an application, by Proposition \ref{blockHerz}, to prove the estimation (\ref{BomgeaK}) above, we need to make that
\begin{align}\label{sumckj}
\Big(\sum\limits_{j=-\infty}^{\infty}|c_{kj}|^p\Big)^{1/p}\lesssim \mathcal C_5,\,\,\textit{\rm for all}\,k\in\mathbb Z.
\end{align}
In fact, we decompose 
\begin{align}\label{K1K2}
\sum\limits_{j=-\infty}^{\infty}|c_{kj}|^p&=\sum\limits_{j= 0}^{\infty}|c_{kj}|^p+\sum\limits_{j=-\infty}^{-1}|c_{kj}|^p\nonumber\\
&:=K_1+K_2.
\end{align}
Since $\sigma>(1-p)/p$ and the H\"{o}lder inequality, we get
\begin{align}\label{K1esti}
K_1&\lesssim \Big(\sum\limits_{j= 0}^{\infty}j^{\sigma}|c_{kj}|\Big)^{p}\nonumber
\\
&=\Big(\sum\limits_{j=0}^{\infty}j^{\sigma}\|\Omega\|_{L^{{q^*}'}(S_{Q-1})}\Big(\int_{(2^{j-1}, 2^j]}\dfrac{\Phi(t)}{t}dt\Big)2^{jQ(\frac{1}{q^*}+\frac{\alpha^*}{Q})}\Big)^{p}\nonumber
\\
&\lesssim \|\Omega\|^p_{L^{{q^*}'}(S_{Q-1})}\Big(\sum\limits_{j=0}^{\infty}\int_{(2^{j-1}, 2^j]}\dfrac{\Phi(t)}{t^{1-\frac{Q}{q^*}-\alpha^*}}({\rm log}_2(t)+1)^{\sigma}dt\Big)^{p}\nonumber
\\
&=\|\Omega\|^p_{L^{{q^*}'}(S_{Q-1})}\Big(\int_{1/2}^{\infty}\dfrac{\Phi(t)}{t^{1-\frac{Q}{q^*}-\alpha^*}}({\rm log}_2(t)+1)^{\sigma}dt\Big)^{p}.
\end{align}
By estimating as above, we also have
\begin{align}\label{K2esti}
K_2&\lesssim \Big(\sum\limits_{j=-\infty}^{-1}|j|^{\sigma}|c_{kj}|\Big)^{p}\nonumber
\\
&=\Big(\sum\limits_{j=-\infty}^{-1}|j|^{\sigma}\|\Omega\|_{L^{{q^*}'}(S_{Q-1})}\Big(\int_{(2^{j-1}, 2^j]}\dfrac{\Phi(t)}{t}dt\Big)2^{jQ\frac{(\delta-1)}{\delta}(\frac{1}{q^*}+\frac{\alpha^*}{Q})}\Big)^{p}\nonumber
\\
&\lesssim \|\Omega\|^p_{L^{{q^*}'}(S_{Q-1})}\Big(\sum\limits_{j=-\infty}^{-1}\int_{(2^{j-1}, 2^j]}\dfrac{\Phi(t)}{t^{1-\frac{Q(\delta-1)}{\delta}(\frac{1}{q^*}+\frac{\alpha^*}{Q})}}|{\rm log}_2(t)|^{\sigma}dt\Big)^{p}\nonumber
\\
&=\|\Omega\|^p_{L^{{q^*}'}(S_{Q-1})}\Big(\int_{0}^{1/2}\dfrac{\Phi(t)}{t^{1-\frac{Q(\delta-1)}{\delta}(\frac{1}{q^*}+\frac{\alpha^*}{Q})}}|{\rm log}_2(t)|^{\sigma}dt\Big)^{p}.
\end{align}
Therefore, by (\ref{K1K2})-(\ref{K2esti}), we obtain 
\begin{align}
\Big(\sum\limits_{j=-\infty}^{\infty}|c_{kj}|^p\Big)^{1/p}&\lesssim \|\Omega\|_{L^{{q^*}'}(S_{Q-1})}\Big\{\Big(\int_{1/2}^{\infty}\dfrac{\Phi(t)}{t^{1-\frac{Q}{q^*}-\alpha^*}}({\rm log}_2(t)+1)^{\sigma}dt\Big)^{p}+\nonumber
\\
&\,\,\,\,\,\,\,\,\,\,\,\,\,\,\,\,\,\,\,\,\,\,\,\,\;\;\;\;\;\;\;\;\;\;\;\;\;\;\;\;\;+\Big(\int_{0}^{1/2}\dfrac{\Phi(t)}{t^{1-\frac{Q(\delta-1)}{\delta}(\frac{1}{q^*}+\frac{\alpha^*}{Q})}}|{\rm log}_2(t)|^{\sigma}dt\Big)^{p}\Big\}^{1/p}\nonumber\\
&\simeq \mathcal C_5.\|\Omega\|_{L^{{q^*}'}(S_{Q-1})},\nonumber
\end{align}
which implies the inequality (\ref{sumckj}) above. Thus, the proof of theorem is completed.
\end{proof}
\begin{theorem}\label{TheoremHerz3}
Let $0<p<1\leq q<\infty$, $\alpha\in\mathbb R^+$, $\Omega\in L^{{q}'}(S_{Q-1})$, $\sigma>(1-p)/p$ and $\omega(x)=|x|^{\gamma}$ with $\gamma\in (-Q,0]$. If the following condition holds,
$$
\mathcal C_6=\int_{1}^{\infty}\dfrac{\Phi(t)}{t^{1-\frac{Q+\gamma}{q}-\frac{(Q+\gamma)\alpha}{Q}}}({\rm log}_2(t)+1)^{\sigma}dt +\int_{0}^{1}\dfrac{\Phi(t)}{t^{1-\frac{Q+\gamma}{q}-\frac{(Q+\gamma)\alpha}{Q}}}|{\rm log}_2(t)|^{\sigma}dt<\infty,
$$
then ${\mathcal H}_{\Phi,\Omega}$ is bounded from ${\mathop{K}\limits^.}^{\alpha,p}_{q,\omega}(\mathbb H^n)$ to itself.
\end{theorem}
\begin{proof}
Let us recall that ${\mathcal B}_{\Phi,\Omega}$ and $b_k$ are defined as in the proof of Theorem \ref{TheoremHerz2}. Then, we need to prove that
\begin{align}\label{wideckjhkj}
{\mathcal B}_{\Phi,\Omega}(b_k)=\sum\limits_{j=-\infty}^{\infty} {\widetilde c}_{kj}.{\widetilde h}_{kj},
\end{align}
where every ${\widetilde h}_{kj}$ is central $(\alpha,q,\omega)$-block, and 
\begin{align}\label{wideckj}
\Big(\sum\limits_{j=-\infty}^{\infty}|{\widetilde c}_{kj}|^p\Big)^{1/p}\lesssim \mathcal C_6,\,\,\textit{\rm for all}\,k\in\mathbb Z.
\end{align}
By setting $g_{kj}$ as in (\ref{sumgkj}), we will have the relation (\ref{suppgkj}). Thus, by a similar reason as (\ref{HfLebMorrey}) above, we also have
\begin{align}
\|g_{kj}\|_{L^q_\omega(\mathbb H^n)}&=\|g_{kj}\|_{L^q_\omega(U_{k+j})}\nonumber\\
&\lesssim \|\Omega\|_{L^{q'}(S_{Q-1})}\int_{(2^{j-1},2^j]}\dfrac{\Phi(t)}{t^{1-\frac{Q+\gamma}{q}}}\|b_k\|_{L^q_\omega(\delta_{t^{-1}}U_{k+j})}dt.\nonumber
\end{align}
For $t\in (2^{j-1},2^j]$, by applying ${\rm supp}(b_k)\subset U_k$ and $\|b_k\|_{L^q_\omega(U_k)}\leq \omega(U_k)^{-\alpha/Q}$, it follows that
$$
\|b_k\|_{L^q_\omega(\delta_{t^{-1}}U_{k+j})}\leq \|b_k\|_{L^q_\omega(U_{k+1})}=\|b_k\|_{L^q_\omega(U_k)}\leq \omega(U_k)^{-\alpha/Q},
$$
Furthermore, by (\ref{omegaBr}), we have $\frac{\omega(U_{k+j})}{\omega(U_k)}\simeq 2^{j(Q+\gamma)}$. Therefore,
\begin{align}
\|g_{kj}\|_{L^q_\omega(\mathbb H^n)}&\lesssim \|\Omega\|_{L^{q'}(S_{Q-1})}\Big(\int_{(2^{j-1},2^j]}\dfrac{\Phi(t)}{t^{1-\frac{Q+\gamma}{q}}}dt\Big).\Big(\dfrac{\omega(U_{k+j})}{\omega(U_k)}\Big)^{\alpha/Q}\omega(U_{k+j})^{-\alpha/Q}\nonumber
\\
&\lesssim \|\Omega\|_{L^{q'}(S_{Q-1})}\Big(\int_{(2^{j-1},2^j]}\dfrac{\Phi(t)}{t^{1-\frac{Q+\gamma}{q}}}dt\Big).2^{j(Q+\gamma)\alpha/Q}\omega(U_{k+j})^{-\alpha/Q}.\nonumber
\end{align}
Now, for simplicity  we write
 $$\widetilde{c}_{kj}=\|\Omega\|_{L^{q'}(S_{Q-1})}\Big(\int_{(2^{j-1},2^j]}\dfrac{\Phi(t)}{t^{1-\frac{Q+\gamma}{q}}}dt\Big).2^{j(Q+\gamma)\alpha/Q}$$
and
\begin{align}
\widetilde{h}_{kj}=\left\{ \begin{array}{l}
\dfrac{g_{kj}}{\widetilde{c}_{kj}},\,\,\,\,\,\,\,\,\,\,\,\textit{\rm if }\,\widetilde{c}_{kj}\neq 0,
\\
\\
g_{kj}=0,\,\,\textit{\rm if }\,\widetilde{c}_{kj}= 0.\nonumber
\end{array} \right.
\end{align}
Thus, by (\ref{sumgkj}), it is easy to get the decomposition (\ref{wideckjhkj}), and each $\widetilde{h}_{kj}$ is a central $(\alpha, q,\omega)$-block with ${\rm supp}(\widetilde{h}_{kj})\subset U_{k+j}$ and $\|\widetilde{h}_{kj}\|_{L^{q}_\omega(\mathbb H^n)}\lesssim \omega(U_{k+j})^{-\alpha/Q}.$
\vskip 5pt
Next, by $\sigma>(1-p)/p$ and the H\"{o}lder inequality, we will infer as follow 
\begin{align}
\sum\limits_{j=-\infty}^{\infty}|{\widetilde c}_{kj}|^p &\lesssim \Big(\sum\limits_{j=-\infty}^{\infty}|j|^{\sigma}|\widetilde{c}_{kj}|\Big)^{p}\nonumber
\\
&=\Big(\sum\limits_{j=-\infty}^{\infty}|j|^{\sigma}\|\Omega\|_{L^{q'}(S_{Q-1})}\Big(\int_{(2^{j-1},2^j]}\dfrac{\Phi(t)}{t^{1-\frac{Q+\gamma}{q}}}dt\Big).2^{j(Q+\gamma)\alpha/Q}\Big)^p\nonumber
\\
&\lesssim \|\Omega\|^p_{L^{q'}(S_{Q-1})}\Big(\sum\limits_{j=1}^{\infty}\int_{(2^{j-1},2^j]}\dfrac{\Phi(t)}{t^{1-\frac{Q+\gamma}{q}}}|j|^{\sigma}2^{j(Q+\gamma)\alpha/Q}dt +\nonumber
\\
&\,\,\,\,\,\,\,\,\,\,\,\,\,\,\,\,\,\,\,\,\,\,\,\,\,\,\,\,\,\,\,\,\,\,\,\,\,\,\,\,\,\,\,\,\,\;\;\;\;\;\;\;\;\;\;\;\;+\sum\limits_{j=-\infty}^{0}\int_{(2^{j-1},2^j]}\dfrac{\Phi(t)}{t^{1-\frac{Q+\gamma}{q}}}|j|^{\sigma}2^{j(Q+\gamma)\alpha/Q}dt\Big)^p.\nonumber
\end{align}
From this, we have
\begin{align}
&\Big(\sum\limits_{j=-\infty}^{\infty}|{\widetilde c}_{kj}|^p\Big)^{1/p}\lesssim \|\Omega\|_{L^{q'}(S_{Q-1})}\Big(\sum\limits_{j=1}^{\infty}\int_{(2^{j-1},2^j]}\dfrac{\Phi(t)}{t^{1-\frac{Q+\gamma}{q}-\frac{(Q+\gamma)\alpha}{Q}}}({\rm log}_2(t)+1)^{\sigma}dt\,+ \nonumber
\\
&\,\,\,\,\,\,\,\,\,\,\,\,\,\,\,\,\,\,\,\,\,\,\,\,\,\,\,\,\,\,\,\,\,\,\,\,\,\,\,\,\,\,\,\,\,\,\,\,\,\,\,\,\,\,\,\,\,\,\,\,\,\,\,\,\,\,\,\,\,\,\,\,\,\,\,\,\,\,\,\,\,\,\,\,\,\,\,\,\,\,\,\,\,\,\,\,+\sum\limits_{j=-\infty}^{0}\int_{(2^{j-1},2^j]}\dfrac{\Phi(t)}{t^{1-\frac{Q+\gamma}{q}-\frac{(Q+\gamma)\alpha}{Q}}}|{\rm log}_2(t)|^{\sigma}dt\Big)\nonumber
\\
&\lesssim \|\Omega\|_{L^{q'}(S_{Q-1})}\Big(\int_{1}^{\infty}\dfrac{\Phi(t)}{t^{1-\frac{Q+\gamma}{q}-\frac{(Q+\gamma)\alpha}{Q}}}({\rm log}_2(t)+1)^{\sigma}dt +\int_{0}^{1}\dfrac{\Phi(t)}{t^{1-\frac{Q+\gamma}{q}-\frac{(Q+\gamma)\alpha}{Q}}}|{\rm log}_2(t)|^{\sigma}dt\Big).\nonumber
\end{align}
This finishes the proof of the inequality (\ref{wideckj}). Moreover, by Proposition \ref{blockHerz}, we obtain
\begin{align}
\|{\mathcal H}_{\Phi,\Omega}(f)\|_{{\mathop{K}\limits^.}^{\alpha,p}_{q,\omega}(\mathbb H^n)}\lesssim \mathcal C_6.\|\Omega\|_{L^{q'}(S_{Q-1})}.\|f\|_{{\mathop{K}\limits^.}^{\alpha,p}_{q,\omega}(\mathbb H^n)}.\nonumber
\end{align}
Therefore, the proof of the theorem is completed.
\end{proof}
\begin{theorem}\label{TheoremMorreyHerz}
Let $0<p<\infty$, $1<q<\infty$, $\gamma\in\mathbb R$, $\lambda>0$, $\omega(x)=|x|_h^{\gamma}$ and $\Omega\in L^{q'}(S_{Q-1})$. Then, ${\mathcal H}_{\Phi,\Omega}$ is  a bounded operator from ${MK}^{\alpha,\lambda}_{p, q, \omega}(\mathbb H^n)$ to itself if and only if 
$$
\mathcal C_{7}= \int_{0}^{\infty} {\dfrac{\Phi(t)}{t^{1-\alpha-\frac{Q+\gamma}{q}+\lambda}}}dt <  + \infty.
$$
Furthermore, ${\big\|{\mathcal H}_{\Phi,\Omega}\big\|_{{MK}^{\alpha,\lambda}_{p, q, \omega}(\mathbb H^n)\to {MK}^{\alpha,\lambda}_{p, q, \omega}(\mathbb H^n)}\simeq \mathcal C_{7}}.\|\Omega\|_{L^{q'}(S_{Q-1})}.$
\end{theorem}
\begin{proof}
We begin with proving the sufficient condition of this theorem. Recall the estimation (\ref{HfMHerzchik}) that
\begin{align}
\|{\mathcal H}_{\Phi,\Omega}(f)\chi_k\|_{L^q_\omega(\mathbb H^n)}\leq \|\Omega\|_{L^{q'}(S_{Q-1})}\int_{0}^{\infty}\dfrac{\Phi(t)}{t^{1-\frac{Q+\gamma}{q}}}\Big(\|f\chi_{k+\ell-1}\|_{L^q_\omega(\mathbb H^n)}+\|f\chi_{k+\ell}\|_{L^q_\omega(\mathbb H^n)}\Big)dt,\nonumber
\end{align}
where $\ell(t)\in\mathbb Z$ such that $2^{\ell(t)-1}< t^{-1}\leq 2^{\ell(t)}$. On the other hand, from the definition of the Morrey-Herz space, we get
\begin{align}\label{esLqMH}
\|f\chi_k\|_{L^q_\omega(\mathbb H^n)}\leq 2^{k(\lambda-\alpha)}\|f\|_{{MK}^{\alpha,\lambda}_{p, q, \omega}(\mathbb H^n)},\,\,\textit{\rm for all}\,k\in\mathbb Z.
\end{align}
By $t^{-1}\in (2^{\ell-1},2^{\ell}]$ and (\ref{esLqMH}), we imply
\begin{align}
\Big(\|f\chi_{k+\ell-1}\|_{L^q_\omega(\mathbb H^n)}+\|f\chi_{k+\ell}\|_{L^q_\omega(\mathbb H^n)}\Big)&\lesssim 2^{(k+\ell)(\lambda-\alpha)}\|f\|_{{MK}^{\alpha,\lambda}_{p, q, \omega}(\mathbb H^n)}\nonumber
\\
&\lesssim t^{\alpha-\lambda}2^{k(\lambda-\alpha)}\|f\|_{{MK}^{\alpha,\lambda}_{p, q, \omega}(\mathbb H^n)}.\nonumber
\end{align}
Thereforce, we infer
\begin{align}
\|{\mathcal H}_{\Phi,\Omega}(f)\chi_k\|_{L^q_\omega(\mathbb H^n)}\lesssim \|\Omega\|_{L^{q'}(S_{Q-1})}\Big(\int_{0}^{\infty}\dfrac{\Phi(t)}{t^{1-\frac{Q+\gamma}{q}-\alpha+\lambda}}dt\Big)2^{k(\lambda-\alpha)}\|f\|_{{MK}^{\alpha,\lambda}_{p, q, \omega}(\mathbb H^n)}.\nonumber
\end{align}
As a consequence, by $\lambda>0$, we deduce
\begin{align}
&\|{\mathcal H}_{\Phi,\Omega}(f)\|_{{MK}^{\alpha,\lambda}_{p, q, \omega}(\mathbb H^n)}=\mathop{\rm sup}\limits_{k_0\in\mathbb Z}2^{-k_0\lambda}\Big(\sum\limits_{k=-\infty}^{k_0}2^{k\alpha p}\|{\mathcal H}_{\Phi,\Omega}(f)\chi_k\|^p_{L^q_\omega(\mathbb H^n)}\Big)^{1/p}\nonumber
\\
&\,\,\,\lesssim \mathcal C_7.\|\Omega\|_{L^{q'}(S_{Q-1})}\Big\{\mathop{\rm sup}\limits_{k_0\in\mathbb Z}2^{-k_0\lambda}\Big(\sum\limits_{k=-\infty}^{k_0}2^{k\alpha p}2^{k(\lambda-\alpha)p}\Big)^{1/p}\Big\}\|f\|_{{MK}^{\alpha,\lambda}_{p, q, \omega}(\mathbb H^n)}\nonumber
\\
&\,\,\,\lesssim \mathcal C_7.\|\Omega\|_{L^{q'}(S_{Q-1})}.\|f\|_{{MK}^{\alpha,\lambda}_{p, q, \omega}(\mathbb H^n)},\nonumber
\end{align}
which shows that ${\mathcal H}_{\Phi,\Omega}$ is  a bounded operator from ${MK}^{\alpha,\lambda}_{p, q, \omega}(\mathbb H^n)$ to itself.
\\

To make the proof for the necessary condition, let us now take
$$
f(x)=|x|_h^{-\alpha-\frac{Q+\gamma}{q}+\lambda}|\Omega(\delta_{|x|_h^{-1}}.x)|^{q'-2}\overline{\Omega}(\delta_{|x|_h^{-1}}.x),\,\textit{\rm for all}\,\,x\neq 0.
$$
Then, we have
\begin{align}
\|f\chi_k\|_{L^q_\omega(\mathbb H^n)}&=\Big(\int_{C_k}|x|_h^{-\alpha q-(Q+\gamma)+\lambda q}|\Omega(\delta_{|x|_h^{-1}}x)|^{q'}\omega(x)dx\Big)^{1/q}\nonumber
\\
&=\Big(\int_{2^{k-1}}^{2^k}\int_{S_{Q-1}}|\delta_ry'|_h^{-\alpha q-(Q+\gamma)+\lambda q+\gamma}|\Omega(y')|^{q'}r^{Q-1}dy'dr\Big)^{1/q}\nonumber
\\
&=\Big(\int_{2^{k-1}}^{2^k}\int_{S_{Q-1}}r^{-(\alpha-\lambda)q-1}|\Omega(y')|^{q'}dy'dr\Big)^{1/q}.\nonumber
\end{align}
This deduces that
\begin{align}
\|f\chi_k\|_{L^q_\omega(\mathbb H^n)} &= \left\{ \begin{array}{l}
2^{-(\alpha-\lambda)k}\Big(\dfrac{2^{q(\alpha-\lambda)-1}}{q(\alpha-\lambda)}\Big)^{1/q}.\|\Omega\|_{L^q(S_{Q-1})}^{q'/q},\,\,\textit{\rm if}\,\,\alpha \neq\lambda,\nonumber
\\
\\
{\rm log}(2).\|\Omega\|_{L^q(S_{Q-1})}^{q'/q},\,\,\,\,\,\,\,\,\,\,\,\,\,\,\,\,\,\,\,\,\,\,\,\,\,\,\,\,\,\,\,\,\,\,\,\,\,\,\,\,\,\,\,\,\,\textit{\rm otherwise}.\nonumber\\
\end{array} \right.\\
&\simeq 2^{-(\alpha-\lambda)k}.\|\Omega\|_{L^q(S_{Q-1})}^{q'/q}.\nonumber
\end{align}
From the definition of the Morrey-Herz space, we have
\begin{align}\label{fMHerznorm}
\|f\|_{M^{\alpha,\lambda}_{p, q,\omega}(\mathbb H^n)}&\simeq\mathop{\rm sup}\limits_{k_0\in\mathbb Z}2^{-k_0\lambda}\Big(\sum\limits_{k=-\infty}^{k_0}2^{k\alpha p}\Big(2^{-(\alpha-\lambda)k}\|\Omega\|_{L^q(S_{Q-1})}^{q'/q}\Big)^p\Big)^{1/p}
\nonumber
\\
&=\|\Omega\|_{L^q(S_{Q-1})}^{q'/q}\mathop{\rm sup}\limits_{k_0\in\mathbb Z}2^{-k_0\lambda}\Big(\sum\limits_{k=-\infty}^{k_0}2^{k\lambda p}\Big)^{1/p}\simeq \|\Omega\|_{L^q(S_{Q-1})}^{q'/q}.
\end{align}
By choosing $f$ and having the relation $\delta_{|\delta_{t^{-1}|x|_h}y'|^{-1}_h}\delta_{t^{-1}|x|_h}y'=y'$, we obtain
\begin{align}
{\mathcal H}_{\Phi,\Omega}(f)(x)&=\int_{0}^{\infty}\int_{S_{Q-1}}\dfrac{\Phi(t)}{t}\Omega(y')
|\delta_{t^{-1}|x|_h}y'|^{-\alpha-\frac{Q+\gamma}{q}+\lambda}
|\Omega(y')|^{q'-2}\overline{\Omega}(y')dy'dt.\nonumber
\\
&=\|\Omega\|^{q'}_{L^{q'}(S_{Q-1})}.\Big(\int_{0}^{\infty}\dfrac{\Phi(t)}{t^{1-\alpha-\frac{Q+\gamma}{q}+\lambda}}dt\Big).|x|_h^{-\alpha-\frac{Q+\gamma}{q}+\lambda}\nonumber\\
&:=\mathcal C_7.\|\Omega\|^{q'}_{L^{q'}(S_{Q-1})}.g(x).\nonumber
\end{align}
Note that, by a similar argument as (\ref{fMHerznorm}), we immediately have 
$$\|g\|_{M^{\alpha,\lambda}_{p, q,\omega}(\mathbb H^n)}\simeq 1.
$$
From this, by (\ref{fMHerznorm}), we obtain
\begin{align}
\|\mathcal H_{\Phi,\Omega}(f)\|_{M^{\alpha,\lambda}_{p, q,\omega}(\mathbb H^n)}&=\mathcal C_7.\|\Omega\|^{q'}_{L^{q'}(S_{Q-1})}.\|g\|_{M^{\alpha,\lambda}_{p, q,\omega}(\mathbb H^n)}\nonumber
\\
&\mathcal C_7.\simeq \|\Omega\|^{q'}_{L^{q'}(S_{Q-1})}.\dfrac{\|f\|_{M^{\alpha,\lambda}_{p, q,\omega}(\mathbb H^n)}}{\|\Omega\|_{L^q(S_{Q-1})}^{q'/q}}\nonumber\\
&=\mathcal C_7. \|\Omega\|_{L^{q'}(S_{Q-1})}.\|f\|_{M^{\alpha,\lambda}_{p, q,\omega}(\mathbb H^n)}.\nonumber
\end{align}
Therefore, the proof of theorem is achieved.
\end{proof}
\section{The main results about the boundness of ${\mathcal{H}}^b_{\Phi,\Omega}$}\label{4}
Before giving our main results in this section, we need to prove the following useful lemmas.
\begin{lemma}\label{LemmaCMO1}
Let $1\leq q<\infty$, $1<q_1,r_1<\infty$, $-Q<\gamma<\frac{Q}{r'_1-1}$ and $\omega(x)=|x|_h^{\gamma}$. Assume that $\omega\in L^{q'}(S_{Q-1})$, $b\in CMO_\omega^{r_1}(\mathbb H^n)$ and the following condition is true:
\begin{align}
\dfrac{1}{q}=\dfrac{1}{q_1}+\dfrac{1}{r_1}.\nonumber
\end{align}
Then, for any $R>0$, we have
\begin{align}
\|{\mathcal{H}}^b_{\Phi,\Omega}(f)\|_{L^q_\omega(B_R)}\lesssim \|\Omega\|_{L^{q'}(S_{Q-1})}\|b\|_{CMO^{r_1}_\omega(\mathbb H^n)}R^{\frac{Q+\gamma}{r_1}}.\int_{0}^{\infty}\dfrac{\Phi(t)}{t^{1-\frac{Q+\gamma}{q_1}}}(2+\Psi(t))\|f\|_{L^{q_1}_\omega(B_{t^{-1}R})}dt,\nonumber
\end{align}
where $\Psi(t)=t^{-Q}\chi_{(0,1]}(t)+ t^{Q}\chi_{(1,\infty)}(t).$
\end{lemma}
\begin{proof}
By making the Minkowski inequality and the H\"{o}lder inequality, we have
\begin{align}
&\|{\mathcal{H}}^b_{\Phi,\Omega}(f)\|_{L^q_\omega(B_R)}\nonumber
\\
&\leq \int_{0}^{\infty}\dfrac{\Phi(t)}{t}\int_{S_{Q-1}}|\Omega(y')|\Big(\int_{B_R}|b(x)-b(\delta_{{t^{-1}}|x|_h}y')|^q.|f(\delta_{t^{-1}|x|_h}y')|^q\omega dx\Big)^{1/q}dy'dt
\nonumber
\\
&\leq \int_{0}^{\infty}\dfrac{\Phi(t)}{t}\int_{S_{Q-1}}|\Omega(y')|. \|b(\cdot)-b(\delta_{t^{-1}|\cdot|_h}y')\|_{L^{r_1}_{\omega}(B_R)}.\|f(\delta_{t^{-1}|\cdot|_h}y')\|_{L^{q_1}_{\omega}(B_R)}dy'dt
\nonumber.
\end{align}
Thus, by using the H\"{o}lder inequality again, we infer
\begin{align}\label{HfMorCMO}
\|{\mathcal{H}}^b_{\Phi,\Omega}(f)\|_{L^q_\omega(B_R)}&\leq \|\Omega\|_{L^{q'}(S_{Q-1})}\int_{0}^{\infty}\dfrac{\Phi(t)}{t}\Big(\int_{S_{Q-1}}\|b(\cdot)-b(\delta_{t^{-1}|\cdot|_h}y')\|^{r_1}_{L^{r_1}_{\omega}(B_R)}dy'\Big)^{\frac{1}{r_1}}\times
\nonumber
\\
&\,\,\,\,\,\,\;\;\;\;\;\;\;\;\;\;\;\;\;\;\;\;\;\;\;\;\times \Big(\int_{S_{Q-1}}\|f(\delta_{t^{-1}|x|_h}y')\|^{q_1}_{L^{q_1}_{\omega}(B_R)}dy'\Big)^{\frac{1}{q_1}}dt.
\end{align}
Now, we need to show that
\begin{align}\label{esCMO}
\Big(\int_{S_{Q-1}}\|b(\cdot)-b(\delta_{t^{-1}|\cdot|_h}y')\|^{r_1}_{L^{r_1}_{\omega}(B_R)}dy'\Big)^{\frac{1}{r_1}}\lesssim \Big(2+\Psi(t)\Big)R^{\frac{Q+\gamma}{r_1}}\|b\|_{ {{CMO}}^{r_1}_{\omega}(\mathbb H^n)}.
\end{align}
Infact, it is not hard to see that
\begin{align}\label{I123}
&\Big(\int_{S_{Q-1}}\|b(\cdot)-b(\delta_{t^{-1}|\cdot|_h}y')\|^{r_1}_{L^{r_1}_{\omega}(B_R)}dy'\Big)^{\frac{1}{r_1}}\leq\Big(\int_{S_{Q-1}}\|b(\cdot)-b_{B_{R}}\|^{r_1}_{L^{r_1}_{\omega}(B_R)}dy'\Big)^{\frac{1}{r_1}}+\nonumber
\\
&+ \Big(\int_{S_{Q-1}}\|b_{B_R}-b_{B_{t^{-1}R}}\|^{r_1}_{L^{r_1}_{\omega}(B_R)}dy'\Big)^{\frac{1}{r_1}}+ \Big(\int_{S_{Q-1}}\|b_{B_{t^{-1}R}}-b({\delta_{t^{-1}|\cdot|_h}y'})\|^{r_1}_{L^{r_1}_{\omega}(B_R)}dy'\Big)^{\frac{1}{r_1}}\nonumber
\\
&:= I_1+I_2+I_3.
\end{align}
By the definition of the space ${{CMO}}_{\omega}^{r_1}(\mathbb H^n)$, we may estimate $I_1$ as follows
\begin{align}\label{esI1}
I_1&=\Big(\omega(B_R)\int_{S_{Q-1}}\frac{1}{\omega(B_R)}\|b(\cdot)-b_{B_{R}}\|^{r_1}_{L^{r_1}_{\omega}(B_R)}dy'\Big)^{\frac{1}{r_1}}\nonumber\\
&\lesssim R^{\frac{Q+\gamma}{r_1}}\|b\|_{{{CMO}}_{\omega}^{r_1}(\mathbb H^n)}.
\end{align}
Similarly, we have
\begin{align}\label{I2CMO}
I_2&\lesssim \omega(B_R)^{{\frac{1}{r_1}}}.|b_{B_R}-b_{B_{t^{-1}R}}|.
\end{align}
For $t\leq 1$, by the H\"{o}lder inequality, it follows that
\begin{align}
&|b_{B_R}-b_{t^{-1}R}|\leq \dfrac{1}{|B_R|}\int_{B_{t^{-1}R}}|b(x)-b_{B_{t^{-1}R}}|dx\nonumber
\\
&\leq \dfrac{\omega(B_{t^{-1}R})^{\frac{1}{r_1}}}{|B_R|}\Big(\dfrac{1}{\omega(B_{t^{-1}R})}\int_{B_{t^{-1}R}}|b(x)-b_{B_{t^{-1}R}}|^{r_1}\omega(x)dx\Big)^{\frac{1}{r_1}}\Big(\int_{B_{t^{-1}R}}\omega(x)^{1-r_1'}dx\Big)^{\frac{1}{r_1'}}
\nonumber
\\
&\leq \dfrac{\omega(B_{t^{-1}R})^{\frac{1}{r_1}}}{|B_R|}\Big(\int_{B_{t^{-1}R}}\omega(x)^{1-r_1'}dx\Big)^{\frac{1}{r_1'}}.\|b\|_{{{CMO}}_{\omega}^{r_1}(\mathbb H^n)}.\nonumber
\end{align}
Note that, by (\ref{omegaBr}) and $\gamma\in (-Q,\frac{Q}{r_1'-1}
)$, we get
$$
\dfrac{\omega(B_{t^{-1}R})^{\frac{1}{r_1}}}{|B_R|}\Big(\int_{B_{t^{-1}R}}\omega(x)^{1-r_1'}dx\Big)^{\frac{1}{r_1'}}\simeq\dfrac{(t^{-1}R)^{\frac{Q+\gamma}{r_1}}}{R^Q}(t^{-1}R)^{\frac{\gamma(1-r_1')+Q}{r_1'}}=t^{-Q}.
$$
This implies  $|b_{B_R}-b_{t^{-1}R}|\lesssim t^{-Q}\|b\|_{{{CMO}}_{\omega}^{r_1}(\mathbb H^n)}$. In the case $t>1$, by estimating as above, we also get $|b_{B_R}-b_{t^{-1}R}|\lesssim t^Q\|b\|_{{{CMO}}_{\omega}^{r_1}(\mathbb H^n)}$. Therefore, by (\ref{I2CMO}), we obtain that
\begin{align}\label{esI2}
I_2\lesssim \omega(B_R)^{\frac{1}{r_1}}\Psi(t).\|b\|_{{{CMO}}_{\omega}^{r_1}(\mathbb H^n)}\lesssim R^{\frac{Q+\gamma}{r_1}}\Psi(t).\|b\|_{{{CMO}}_{\omega}^{r_1}(\mathbb H^n)}.
\end{align}
Next, we have 
\begin{align}\label{I3CMO}
I_3^{r_1} \leq &\int_{S_{Q-1}}\int_{B_R}|b_{B_{t^{-1}R}}-b(\delta_{t^{-1}|x|_h}y')|^{r_1}\omega(x)dxdy'\nonumber
\\
&=\int_{S_{Q-1}}\int_{0}^R\int_{S_{Q-1}}|b_{B_{t^{-1}R}}-b(\delta_{t^{-1}r}y')|^{r_1}r^{Q+\gamma-1}du'drdy'
\nonumber
\\
&\lesssim t^{Q+\gamma}\int_{0}^{t^{-1}R}\int_{S_{Q-1}}|b_{B_{t^{-1}R}}-b(\delta_{z}y')|^{r_1}z^{Q+\gamma-1}dy'dz\nonumber
\\
&=t^{Q+\gamma}\int_{B_{t^{-1}R}}|b(x)-b_{B_{t^{-1}R}}|^{r_1}\omega(x)dx.
\end{align}
This deduces that
\begin{align}
I_3\lesssim t^{\frac{Q+\gamma}{r_1}}\omega(B_{t^{-1}R})^{\frac{1}{r_1}}\|b\|_{{{CMO}}_{\omega}^{r_1}(\mathbb H^n)}\lesssim R^{\frac{Q+\gamma}{r_1}}\|b\|_{{{CMO}}_{\omega}^{r_1}(\mathbb H^n)}.\nonumber
\end{align}
Consequently, by (\ref{esI1}) and (\ref{esI2}), the inequality (\ref{esCMO}) is true. On the other hand, by (\ref{S0BR}), we have
\begin{align}
 \Big(\int_{S_{Q-1}}\|f(\delta_{t^{-1}|x|_h}y')\|^{q_1}_{L^{q_1}_{\omega}(B_R)}dy'\Big)^{\frac{1}{q_1}}\lesssim t^{\frac{Q+\gamma}{q_1}}\|f\|_{L^{q_1}_\omega(B_{t^{-1}R})}.\nonumber
\end{align}
Therefore, by (\ref{HfMorCMO}) and (\ref{esCMO}), the proof of this lemma is completed.
\end{proof}
\begin{lemma}\label{LemmaCMO2}
Let $1\leq q, q^*_1, r^*_1<\infty$, $1\leq\zeta\leq r^*_1$, 
$\omega\in A_{\zeta}$ with the finite
critical index $r_\omega$ for the reverse H\"{o}lder condition. Assume that $\omega\in L^{q'}(S_{Q-1})$, $b\in CMO_\omega^{r^*_1}(\mathbb H^n)$ and the following condition is true:
\begin{align}\label{r*q*}
\dfrac{1}{q}>\Big(\dfrac{1}{q^*_1}+\dfrac{1}{r^*_1}\Big)\zeta\dfrac{r_\omega}{r_\omega-1}.
\end{align}
Then, for any $R>0$, we have
\begin{align}
\|{\mathcal{H}}^b_{\Phi,\Omega}(f)\|_{L^q_\omega(B_R)}&\lesssim \|\Omega\|_{L^{q'}(S_{Q-1})}\|b\|_{CMO^{r^*_1}_\omega(\mathbb H^n)}\times
\nonumber
\\
&\,\,\,\,\,\,\,\,\,\,\times\int_{0}^{\infty}\dfrac{\Phi(t)}{t}(2+\Psi(t))\dfrac{\omega(B_R)^{\frac{1}{q}}}{\omega(B_{t^{-1}R})^{\frac{1}{q^*_1}}}\|f\|_{L^{q_1^*}_\omega(B_{t^{-1}R})}dt.\nonumber
\end{align}
\end{lemma}
\begin{proof}
By the inequality (\ref{r*q*}), there exist two real numbers  $r_1, q_1$ such that
$$
\frac{1}{q_1}>\frac{\zeta}{q_1^*}\frac{r_\omega}{r_\omega-1},\frac{1}{r_1}>\frac{\zeta}{r_1^*}\frac{r_\omega}{r_\omega-1},$$
and $$\frac{1}{q_1}+ \frac{1}{r_1}=\frac{1}{q}.$$
Because of having  $\dfrac{1}{q_1}+ \dfrac{1}{r_1}=\dfrac{1}{q}$, we also obtain the inequality (\ref{HfMorCMO}).
\vskip 5pt
Now, we will refine the estimations of $I_1,I_2, I_3$ in Lemma \ref{LemmaCMO1}. From (\ref{esI1}) above and $r_1< r_1^*$, we infer
\begin{align}\label{esI1end}
I_1\leq \omega(B_R)^{\frac{1}{r_1}}\|b\|_{{CMO}_{\omega}^{r_1}(\mathbb H^n)}\leq \omega(B_R)^{\frac{1}{r_1}}\|b\|_{{CMO}_{\omega}^{r_1^*}(\mathbb H^n)}.
\end{align}
To compute $I_2$, we consider the case that $t\leq 1$. Since $\omega\in A_{\zeta}\subset A_{r_1^*}$ and Proposition \ref{pro2.4DFan}, one has
\begin{align}
|b_R-b_{B_{t^{-1}R}}|&\leq \dfrac{1}{|B_R|}\int_{B_{t^{-1}R}}|b(x)-b_{B_{t^{-1}R}}|dx\nonumber\\
&= \dfrac{|B_{t^{-1}R}|}{|B_R|}\Big(\dfrac{1}{|B_{t^{-1}R}|}\int_{B_{t^{-1}R}}|b(x)-b_{B_{t^{-1}R}}|dx\Big)\nonumber
\\
&\leq t^{-Q}\Big(\dfrac{1}{\omega(B_{t^{-1}R})}\int_{B_{t^{-1}R}}|b(x)-b_{B_{t^{-1}R}}|^{r_1^*}\omega(x)dx\Big)^{\frac{1}{r_1^*}}\nonumber
\\
&\leq  t^{-Q}\|b\|_{CMO^{r_1^*}_\omega(\mathbb H^n)}.\nonumber
\end{align}
Otherwise, for $t>1$ we get  $|b_R-b_{B_{t^{-1}R}}|\leq t^{Q}\|b\|_{CMO^{r_1^*}_\omega(\mathbb H^n)}.$ Thus, by (\ref{I2CMO}), we infer
\begin{align}\label{esI2end}
I_{2}\lesssim \omega(B_R)^{^{\frac{1}{r_1}}}\Psi(t).\|b\|_{{{CMO}}_{\omega}^{r_1^*}(\mathbb H^n)}.
\end{align}
Because of having $\frac{1}{r_1}>\frac{1}{r_1^*}\zeta\frac{r_\omega}{r_\omega-1}$, there exists $r\in (1, r_\omega)$ such that $\frac{r_1^*}{\zeta}=r_1.r'$. By using the H\"{o}lder inequality and the reverse H\"{o}lder condition again, we get
\begin{align}
I_{3}^{r_1}&\leq \int_{S_{Q-1}}\Big(\int_{B_R}{|b_{B_{t^{-1}R}}-b(\delta_{t^{-1}|x|_h}y')|^{\frac{r_1^*}{\zeta}}dx\Big)^{\frac{\zeta r_1}{r_1^*}}\Big(\int_{B_R}\omega(x)^{r}dx\Big)^{\frac{1}{r}}dy'}\nonumber
\\
&\lesssim |B_R|^{\frac{-1}{r'}}\omega(B_R)\Big(\int_{S_{Q-1}}\int_{B_R}{|b_{B_{t^{-1}R}}-b(\delta_{t^{-1}|x|_h}y')|^{\frac{r_1^*}{\zeta}}dxdy'}\Big)^{\frac{\zeta r_1}{r_1^*}}.\nonumber
\end{align}
By making as (\ref{I3CMO}) and Proposition \ref{pro2.4DFan}, we have
\begin{align}\label{esI3end}
I_3 &\lesssim  |B_R|^{\frac{-\zeta}{r^*_1}}\omega(B_R)^{\frac{1}{r_1}}\Big(\int_{S_{Q-1}}\int_{B_R}{|b_{B_{t^{-1}R}}-b(\delta_{t^{-1}|x|_h}y')|^{\frac{r_1^*}{\zeta}}dxdy'}\Big)^{\frac{\zeta}{r_1^*}}\nonumber
\\
&\lesssim |B_R|^{\frac{-\zeta}{r^*_1}}|B_{t^{-1}R}|^{\frac{\zeta}{r^*_1}}\omega(B_R)^{\frac{1}{r_1}}\Big(t^Q\dfrac{1}{|B_{t^{-1}R}|}\int_{B_{t^{-1}R}}|b(x)-b_{t^{-1}R}|^{\frac{r_1^*}{\zeta}}dx\Big)^{\frac{\zeta}{r_1^*}}\nonumber
\\
&\lesssim \omega(B_R)^{\frac{1}{r_1}}\Big(\dfrac{1}{\omega(B_{t^{-1}R})}\int_{B_{t^{-1}R}}|b(x)-b_{t^{-1}R}|^{r_1^*}\omega(x)dx\Big)^{\frac{1}{r_1^*}}\nonumber
\\
&\leq \omega(B_R)^{\frac{1}{r_1}}\|b\|_{CMO^{{r_1}^*}_\omega(\mathbb H^n)}.
\end{align}
Thus, by (\ref{I123}), (\ref{esI1end}), (\ref{esI2end}) and (\ref{esI3end}), we obtain  the important inequality as follows 
\begin{align}\label{I123end}
\Big(\int_{S_{Q-1}}\|b(\cdot)-b(\delta_{t^{-1}|\cdot|_h}y')\|^{r_1}_{L^{r_1}_{\omega}(B_R)}dy'\Big)^{\frac{1}{r_1}}\lesssim \omega(B_R)^{\frac{1}{r_1}}(2+\Psi(t)).\|b\|_{CMO^{{r_1}^*}_\omega(\mathbb H^n)}.
\end{align}
Because of making $\frac{1}{q_1}>\frac{1}{q_1^*}\zeta\dfrac{r_\omega}{r_\omega-1}$ and estimating (\ref{esI3end}) above, we imply
\begin{align}
\Big(\int_{S_{Q-1}}\|f(\delta_{t^{-1}|x|_h}y')\|^{q_1}_{L^{q_1}_{\omega}(B_R)}dy'\Big)^{\frac{1}{q_1}}\lesssim\omega(B_R)^{\frac{1}{q_1}}\omega(B_{t^{-1}R})^{\frac{-1}{q^*_1}}\|f\|_{L^{q^*_1}_\omega(B_{t^{-1}R})}.\nonumber
\end{align}
From this, by (\ref{HfMorCMO}) and (\ref{I123end}), the proof of this lemma is finished.
\end{proof}
Now, we are in a position to give the boundedness of the commutators of the rough Hausdorff operators on the weighted Morrey, Herz and Morrey-Herz spaces on the Heisenberg group. The first main result in this section is as follows.
\begin{theorem}\label{MorreyCMO1}
Let $-\frac{1}{q_1}<\lambda<0$. Suppose that the assumptions of Lemma \ref{LemmaCMO1} hold. Then,  if 
$$\mathcal C_8=\int_{0}^{\infty}\dfrac{\Phi(t)}{t^{1+(Q+\gamma)\lambda}}(2+\Psi(t))dt<\infty,
$$
we have ${\mathcal H}^b_{\Phi,\Omega}$ is bounded from ${\mathop B\limits^.}^{q_1,\lambda}_{\omega}(\mathbb H^n)$ to ${\mathop B\limits^.}^{q,\lambda}_{\omega}(\mathbb H^n)$.
\end{theorem}
\begin{proof}
For any $R>0$, by Lemma \ref{LemmaCMO1}, we have
\begin{align}
\dfrac{1}{\omega(B_R)^{\frac{1}{q}+\lambda}}\|{\mathcal{H}}^b_{\Phi,\Omega}(f)\|_{L^q_\omega(B_R)} &\lesssim \|\Omega\|_{L^{q'}(S_{Q-1})}\|b\|_{CMO^{r_1}_\omega(\mathbb H^n)}\Big(\int_{0}^{\infty}\dfrac{\Phi(t)}{t^{1-\frac{Q+\gamma}{q_1}}}(2+\Psi(t))\times\nonumber
\\
&\,\,\,\,\,\,\times \dfrac{R^{\frac{Q+\gamma}{r_1}}\omega(B_{t^{-1}R})^{\frac{1}{q_1}+\lambda}}{\omega(B_{R})^{\frac{1}{q}+\lambda}}\dfrac{1}{\omega(B_{t^{-1}R})^{\frac{1}{q_1}+\lambda}}\|f\|_{L^{q_1}_\omega(B_{t^{-1}R})}dt,\nonumber
\end{align}
By a simple calculation, one has
$$
\dfrac{R^{\frac{Q+\gamma}{r_1}}\omega(B_{t^{-1}R})^{\frac{1}{q_1}+\lambda}}{\omega(B_{R})^{\frac{1}{q}+\lambda}}\simeq \dfrac{R^{\frac{Q+\gamma}{r_1}}(t^{-1}R)^{(Q+\gamma)(\frac{1}{q_1}+\lambda)}}{R^{(Q+\gamma)(\frac{1}{q}+\lambda)}}=t^{-(Q+\gamma)(\frac{1}{q_1}+\lambda)}.
$$
Thus, we obtain 
$$
\|{\mathcal H}^b_{\Phi,\Omega}(f)\|_{{\mathop B\limits^.}^{q,\lambda}_{\omega}(\mathbb H^n)}\lesssim \mathcal C_8.\|\Omega\|_{L^{q'}(S_{Q-1})}\|b\|_{CMO^{r_1}_\omega(\mathbb H^n)}\|f\|_{{\mathop B\limits^.}^{q_1,\lambda}_{\omega}(\mathbb H^n)},
$$
which finishes the proof of theorem.
\end{proof}
\begin{theorem}\label{MorreyCMO2}
Let the assumptions of Lemma \ref{LemmaCMO2} hold and $\lambda\in(-\frac{1}{q_1^*},0)$. If 
$$\mathcal C_9=\int_{0}^{1}\dfrac{\Phi(t)}{t^{1+Q\lambda(\delta-1)/\delta}}(2+\Psi(t))dt+\int_{1}^{\infty}\dfrac{\Phi(t)}{t^{1+Q\zeta\lambda}}(2+\Psi(t))dt<\infty,
$$
then ${\mathcal H}^b_{\Phi,\Omega}$ is bounded from ${\mathop B\limits^.}^{q_1^*,\lambda}_{\omega}(\mathbb H^n)$ to ${\mathop B\limits^.}^{q,\lambda}_{\omega}(\mathbb H^n)$.
\end{theorem}
\begin{proof}
For $R>0$, by making Lemma \ref{LemmaCMO2}, we infer
\begin{align}
\dfrac{1}{\omega(B_R)^{\frac{1}{q}+\lambda}}\|{\mathcal{H}}^b_{\Phi,\Omega}(f)\|_{L^q_\omega(B_R)}&\lesssim \|\Omega\|_{L^{q'}(S_{Q-1})}\|b\|_{CMO^{r^*_1}_\omega(\mathbb H^n)}\Big(\int_{0}^{\infty}\dfrac{\Phi(t)}{t}(2+\Psi(t))\times\nonumber
\\
&\,\,\,\,\,\,\times\Big(\dfrac{\omega(B_{t^{-1}R})}{\omega(B_R)}\Big)^{\lambda}.\dfrac{1}{\omega(B_{t^{-1}R})^{\frac{1}{q^*_1}+\lambda}}\|f\|_{L^{q_1^*}_\omega(B_{t^{-1}R})}dt.\nonumber
\end{align}
Hence, by the inequality (\ref{t-1Rlambda}) and the definition of the Morrey space, we estimate
$$
\|{\mathcal H}^b_{\Phi,\Omega}(f)\|_{{\mathop B\limits^.}^{q,\lambda}_{\omega}(\mathbb H^n)}\lesssim \mathcal C_9.\|\Omega\|_{L^{q'}(S_{Q-1})}\|b\|_{CMO^{r_1^*}_\omega(\mathbb H^n)}\|f\|_{{\mathop B\limits^.}^{q_1^*,\lambda}_{\omega}(\mathbb H^n)}.
$$
This implies that the proof of theorem is ended.
\end{proof}
\begin{theorem}\label{HerzCMO2}
Let $1\leq p, q, q^*_1, r^*_1<\infty$, $1\leq\zeta\leq r^*_1$, $\alpha\in\mathbb R$, $\alpha^*_1<0$ and 
$\omega\in A_{\zeta}$ with the finite critical index $r_\omega$ for the reverse H\"{o}lder condition and $\delta\in (1,r_\omega)$. Assume that $\omega\in L^{q'}(S_{Q-1})$, $b\in CMO_\omega^{r^*_1}(\mathbb H^n)$, the hypothesis (\ref{alpha*alpha}) in Theorem \ref{TheoremHerz1} and the hypothesis (\ref{r*q*}) in Lemma \ref{LemmaCMO2} hold.
\\
$(\rm i)$ If $\dfrac{1}{q^*_1}+\dfrac{\alpha^*_1}{Q}\geq 0$ and 
$$\mathcal C_{10.1}= \int_{0}^{1/2}\dfrac{\Phi(t)}{t^{1-Q\frac{(\delta-1)}{\delta}(\frac{1}{q^*_1}+\frac{\alpha^*_1}{Q})}}(2+\Psi(t))dt +\int_{1/2}^{\infty}\dfrac{\Phi(t)}{t^{1-Q\zeta(\frac{1}{q^*_1}+\frac{\alpha^*_1}{Q})}}(2+\Psi(t))dt<\infty,
$$
then 
$
\|{\mathcal H}^b_{\Phi,\Omega}(f)\|_{{\mathop{K}\limits^.}^{\alpha, p}_{q,\omega}(\mathbb H^n)}\lesssim \mathcal C_{10.1}.\|\Omega\|_{L^{q'}(S_{Q-1})}\|b\|_{CMO^{r_1^*}_\omega(\mathbb H^n)}\|f\|_{{\mathop{K}\limits^.}^{\alpha^*_1, p}_{q_1^*,\omega}(\mathbb H^n)}.
$
\\
$\,(\rm ii)$ If $\dfrac{1}{q^*_1}+\dfrac{\alpha^*_1}{Q}<0$ and 
$$\mathcal C_{10.2}= \int_{0}^{1/2}\dfrac{\Phi(t)}{t^{1-Q\zeta(\frac{1}{q^*_1}+\frac{\alpha^*_1}{Q})}}(2+\Psi(t))dt+\int_{1/2}^{\infty}\dfrac{\Phi(t)}{t^{1-Q\frac{(\delta-1)}{\delta}(\frac{1}{q^*_1}+\frac{\alpha^*_1}{Q})}}(2+\Psi(t))dt<\infty,
$$
then 
$
\|{\mathcal H}^b_{\Phi,\Omega}(f)\|_{{\mathop{K}\limits^.}^{\alpha, p}_{q,\omega}(\mathbb H^n)}\lesssim \mathcal C_{10.2}.\|\Omega\|_{L^{q'}(S_{Q-1})}\|b\|_{CMO^{r_1^*}_\omega(\mathbb H^n)}\|f\|_{{\mathop{K}\limits^.}^{\alpha^*_1, p}_{q_1^*,\omega}(\mathbb H^n)}.
$
\end{theorem}
\begin{proof}
From Lemma \ref{LemmaCMO2}, for any $k\in\mathbb Z$, we get
\begin{align}
\|{\mathcal H}_{\Phi,\Omega}(f)\chi_k\|_{L^{q}_\omega(\mathbb H^n)}&\lesssim \|\Omega\|_{L^{q'}(S_{Q-1})}\|b\|_{CMO^{r_1^*}_\omega(\mathbb H^n)}\int_{0}^{\infty}\dfrac{\Phi(t)}{t}(2+\Psi(t))\times\nonumber
\\
&\,\,\,\,\,\,\,\times\dfrac{\omega(B_k)^{\frac{1}{q}}}{\omega(\delta_{t^{-1}}B_k)^{\frac{1}{q^*_1}}}\|f\|_{L^{q^*_1}_\omega(\delta_{t^{-1}}B_k)}dt,\nonumber
\end{align}
where $\delta_{t^{-1}}B_k=\{z\in\mathbb H^n: |z|_h\leq \frac{2^{k}}{t}\}.$
Note that, by  estimating as the next step proof of Theorem \ref{TheoremHerz1}, we will have the desired results.
\end{proof}
Finally, we have the boundedness of ${\mathcal H}^b_{\Phi,\Omega}$ on the weighted Morrey-Herz spaces on the Heisenberg group as follows.
\begin{theorem}\label{MHerzCMO1}
Let the assumptions of Lemma \ref{LemmaCMO1} hold, $0<p<\infty$, $\lambda >0$ and $\alpha_1=\alpha+\frac{Q+\gamma}{r_1}$. If 
$$\mathcal C_{11}=\int_{0}^{\infty}\dfrac{\Phi(t)}{t^{1-\frac{Q+\gamma}{q_1}-\alpha_1+\lambda}}(2+\Psi(t))dt<\infty,
$$
then ${\mathcal H}^b_{\Phi,\Omega}$ is  a bounded operator from ${MK}^{\alpha_1,\lambda}_{p, q_1, \omega}(\mathbb H^n)$ to ${MK}^{\alpha,\lambda}_{p, q, \omega}(\mathbb H^n)$.
\end{theorem}
\begin{proof}
By making Lemma \ref{LemmaCMO1} and estimating as (\ref{HfMHerzchik}) above, we get
\begin{align}\label{LqLq1CMO}
\|{\mathcal H}^{b}_{\Phi,\Omega}(f)\chi_k\|_{L^q_\omega(\mathbb H^n)}&\lesssim\|\Omega\|_{L^{q'}(S_{Q-1})}\|b\|_{CMO^{r_1}_\omega(\mathbb H^n)}\Big(\int_{0}^{\infty}\dfrac{\Phi(t)}{t^{1-\frac{Q+\gamma}{q_1}}}(2+\Psi(t))2^{\frac{k(Q+\gamma)}{r_1}}\times\nonumber
\\
&\,\,\,\,\,\,\,\times\Big(\|f\chi_{k+\ell-1}\|_{L^{q_1}_\omega(\mathbb H^n)}+\|f\chi_{k+\ell}\|_{L^{q_1}_\omega(\mathbb H^n)}\Big)dt\Big)\,\,\textit{\rm for all}\,k\in\mathbb Z.
\end{align}
Here $\ell=\ell(t)\in\mathbb Z$ such that $2^{\ell(t)-1}< t^{-1}\leq 2^{\ell(t)}$. By (\ref{esLqMH}), we have
\begin{align}
\Big(\|f\chi_{k+\ell-1}\|_{L^{q_1}_\omega(\mathbb H^n)}+\|f\chi_{k+\ell}\|_{L^{q_1}_\omega(\mathbb H^n)}\Big)&\lesssim 2^{(k+\ell)(\lambda-\alpha_1)}\|f\|_{{MK}^{\alpha_1,\lambda}_{p, q_1, \omega}(\mathbb H^n)}\nonumber
\\
&\lesssim t^{\alpha_1-\lambda}2^{k(\lambda-\alpha_1)}\|f\|_{{MK}^{\alpha_1,\lambda}_{p, q_1, \omega}(\mathbb H^n)}.\nonumber
\end{align}
From this, by (\ref{LqLq1CMO}) and definition of $\alpha_1$, we lead to
\begin{align}
\|{\mathcal H}^b_{\Phi,\Omega}(f)\chi_k\|_{L^q_\omega(\mathbb H^n)}
&\lesssim\|\Omega\|_{L^{q'}(S_{Q-1})}\|b\|_{CMO^{r_1}_\omega(\mathbb H^n)}\Big(\int_{0}^{\infty}\dfrac{\Phi(t)}{t^{1-\frac{Q+\gamma}{q_1}-\alpha_1+\lambda}}(2+\Psi(t))dt\Big)\times\nonumber
\\
&\,\,\,\,\,\,\,\times 2^{k(\lambda-\alpha)}\|f\|_{{MK}^{\alpha_1,\lambda}_{p, q_1, \omega}(\mathbb H^n)}.\nonumber
\end{align}
Hence, by $\lambda>0$, one has
\begin{align}
&\|{\mathcal H^b}_{\Phi,\Omega}(f)\|_{{MK}^{\alpha,\lambda}_{p, q, \omega}(\mathbb H^n)}=\mathop{\rm sup}\limits_{k_0\in\mathbb Z}2^{-k_0\lambda}\Big(\sum\limits_{k=-\infty}^{k_0}2^{k\alpha p}\|{\mathcal H}^b_{\Phi,\Omega}(f)\chi_k\|^p_{L^q_\omega(\mathbb H^n)}\Big)^{1/p}\nonumber
\\
&\,\,\,\lesssim \mathcal C_{11}.\|\Omega\|_{L^{q'}(S_{Q-1})} \|b\|_{CMO^{r_1}_\omega(\mathbb H^n)}\Big\{\mathop{\rm sup}\limits_{k_0\in\mathbb Z}2^{-k_0\lambda}\Big(\sum\limits_{k=-\infty}^{k_0}2^{k\alpha p}2^{k(\lambda-\alpha)p}\Big)^{1/p}\Big\}\|f\|_{{MK}^{\alpha_1,\lambda}_{p, q_1, \omega}(\mathbb H^n)}\nonumber
\\
&\,\,\,\lesssim \mathcal C_{11}\|\Omega\|_{L^{q'}(S_{Q-1})}\|b\|_{CMO^{r_1}_\omega(\mathbb H^n)}.\|f\|_{{MK}^{\alpha_1,\lambda}_{p, q_1, \omega}(\mathbb H^n)},\nonumber
\end{align}
which implies that ${\mathcal H}^b_{\Phi,\Omega}$ is  a bounded operator from ${MK}^{\alpha_1,\lambda}_{p, q_1, \omega}(\mathbb H^n)$ to ${MK}^{\alpha,\lambda}_{p, q, \omega}(\mathbb H^n)$. Therefore, the theorem is completely proved.
\end{proof}

\end{document}